\documentclass[a4paper,11pt]{article}
\pdfoutput=1
\usepackage[english]{bricearticle}

\title{Harmonic maps from Kähler manifolds}
\author{%
Brice Loustau\footnote{Heidelberg University and Heidelberg Institute of Theoretical Studies. 69120 Heidelberg, Germany.\newline
E-mail: \url{bloustau@mathi.uni-heidelberg.de}}}
\date{} % No date in title page

\makeatletter
\renewcommand\tableofcontents{%
    \@starttoc{toc}%
}
\makeatother

\hypersetup{
    pdftitle={Harmonic maps from Kähler manifolds},
    pdfauthor={Brice Loustau},
    bookmarksdepth=2
}

\begin{document}

\pdfbookmark[1]{Title and abstract}{Title}
\hypersetup{pageanchor=false}
\begin{titlepage}
\maketitle
\thispagestyle{empty}
\bigskip

\begin{abstract}

This report attempts a clean presentation of the theory of harmonic maps from complex and Kähler manifolds
to Riemannian manifolds. 
After reviewing the theory of harmonic maps between Riemannian manifolds initiated by Eells--Sampson  and the Bochner technique,
we specialize to K\"ahler domains and introduce pluriharmonic maps. We prove a refined Bochner formula due to Siu and Sampson and its main consequences, 
such as the strong rigidity results of Siu. We also recount the applications to symmetric spaces of noncompact type and their relation to Mostow rigidity.
Finally, we explain the key role of this theory for the nonabelian Hodge correspondence %due to Corlette and Simpson 
relating the character variety of a compact Kähler manifold and the moduli space of Higgs bundles.

\bigskip \bigskip \bigskip



\bigskip

\end{abstract}

\bigskip \bigskip
\setcounter{tocdepth}{1}
{\small {\tableofcontents}}
\end{titlepage}
\hypersetup{pageanchor=true}
\addtocounter{page}{1}

\cleardoublepage\phantomsection % So that the hyperref thumbnail is correct
\section*{Introduction}
\addcontentsline{toc}{section}{Introduction}

The theory of harmonic maps between Riemannian manifolds was properly started by Eells and Sampson \cite{MR0164306}. A smooth map $f \colon M \to N$
is \emph{harmonic} if it is a critical point of the energy functional 
\begin{equation}
\mathbf{E}(f) = \frac{1}{2} \int_M \Vert \upd f \Vert^2 \, \upd v\,. 
\end{equation}
Equivalently, $f$ solves the corresponding Euler--Lagrange equation:
\begin{equation}
 \Delta f = 0
\end{equation}
where the \emph{Laplacian} $\Delta$, more commonly denoted $\tau(f)$ and called \emph{tension field}, is a nonlinear operator generalizing the Riemannian Laplacian. 
We recall its definition in \autoref{sec:HarmonicMaps} and explain why $\Delta f = 0$ is equivalent to the harmonicity of the one-form $\upd f \in \Omega^1(M, f^*\upT N)$ in the sense of Hodge theory.

The nonlinear PDE $\Delta f = 0$ is typically determined and Eells and Sampson proved the existence of harmonic maps when the target manifold
$N$ has nonpositive sectional curvature (assuming $M$ and $N$ are compact), using a heat flow technique. 
The success of this technique relies on a Bochner formula for maps between Riemannian manifolds generalizing the classical Bochner--Weitzenböck formulas.
They additionally showed rigidity results under stronger curvature assumptions, 
such as nonnegative Ricci curvature for the domain manifold and negative sectional curvature for the target manifold. 
% These rigidity results again rely entirely on the Bochner formula.

In the seminal paper \cite{MR584075}, Siu adapted Eells--Sampson's Bochner techniques to the setting where $M$ is a Kähler manifold. 
Siu's ideas were further developed by Sampson \cite{MR833809} and Carlson--Toledo \cite{MR1019964}, among others. 
Sampson showed that if $M$ is compact Kähler and $N$ has nonpositive Hermitian sectional curvature, any harmonic map $M \to N$ is \emph{pluriharmonic}, a stronger version of harmonicity. 
Moreover Siu's technique yields a constraint on the relation between the image of $\upd f$ in $\upT N$ and the curvature tensor of $N$, 
leading to strong rigidity results under appropriate curvature assumptions. In particular, there are remarkable consequences when $N = G/K$ is a symmetric space of noncompact type,
that relate to the famous Mostow rigidity theorem.

This theory is key for the nonabelian Hodge correspondence initiated by Hitchin \cite{MR887284} and Donaldson \cite{MR887285}
and generalized by Corlette \cite{MR965220} and Simpson \cite{MR944577, MR1159261, MR1179076, SimpsonModuli}, beautifully relating the $G$-character variety of a compact Kähler manifold $M$ and the moduli space of $G$-Higgs bundles over $M$.

The goal of this report is to give a clean presentation of these ideas and results, 
after developing the relevant mathematical background mostly consisting of standard differential geometry. I hope that it becomes a useful reference to non-experts in the theory of harmonic maps
who may wish to apply it to other areas, such as the study of representations of discrete groups in Lie groups. 
Of course, there are already good references on the subject, such as the excellent book \cite{MR1379330} which I used substantially in the preparation of these notes (especially \autoref{sec:SymmetricSpaces}). I shall list other main references for each topic throughout subsequent sections.

\phantomsection % So that the hyperref link to the Acknowledgments is correct
\subsection*{Acknowledgments}
\addcontentsline{toc}{section}{Acknowledgments}

This report originated in an expository talk that I gave at the workshop \emph{Harmonic maps and rigidity\footnote{\url{https://math.unice.fr/~jtoulisse/conf/sisteron.html}}}
in Sisteron, France in April 2019. 
I am very grateful to the organizers Alexis Gilles, Nicolas Tholozan, and Jérémy Toulisse for this wonderful experience.
I also thank all the workshop participants for stimulating mathematical discussions. 

I gratefully acknowledge research support from the NSF Grant DMS1107367 aka \emph{RNMS\@: GEometric structures And Representation varieties} (the \emph{GEAR Network}).

\newpage
\section{Harmonic maps} \label{sec:HarmonicMaps}

Let  $f \colon M \to N$ a smooth map between Riemannian manifolds.
The main goal of this section is to explain the equivalence:
\begin{equation} \label{eq:HarmonicEquiv}
 \text{$f$ harmonic} \quad \Leftrightarrow \quad \Delta f = 0 \quad \Leftrightarrow \quad \upd^*_\nabla \upd f = 0  \quad \Leftrightarrow \quad\text{$\upd f$ is a harmonic $1$-form}
\end{equation}

This requires introducing the basic tools and notions of differential geometry at play in
the theory of harmonic maps between Riemannian manifolds. There are quite a few good references on the subject, such as \cite{MR0164306}, \cite{MR703510}, 
\cite{MR1363513},  \cite{MR2483364},  \cite{MR1391729},  \cite{MR3098705}, \cite[Chap. 9]{MR3726907}. 
I recommend refering to these for more developments and examples. 

\begin{remark}
Contrary to many authors, we will avoid writing any formulas in local coordinates.
\end{remark}

\begin{remark}
 In all this report, $M$ is assumed connected and orientable, and given a fixed orientation. In particular, the volume density $v_M$ on $M$ can be identified
 to the volume form $\vol_M$. When needed, $M$ will also be assumed compact, \eg{} for the last equivalence in \eqref{eq:HarmonicEquiv}.
\end{remark}

\subsection{Preliminary: connections in vector bundles}

Let $E \to M$ be a smooth vector bundle. We denote by $\Gamma(E)$ the space of smooth sections\footnote{More generally, 
we use $\Gamma(\cdot)$ for the space of sections of any sheaf, but we will (barely) use this language only in \autoref{sec:NAH}.}. 
A \emph{(linear) connection} in $E$ is a $\R$-linear map 
\begin{equation}
\nabla \colon \Gamma(E) \to \Gamma(\upT^*M \otimes E)
\end{equation}
that satisfies the \emph{Leibniz rule} $\nabla(f s) = \upd f \otimes s + f \nabla s$. 
If $X \in \Gamma(\upT M)$ is any vector field,
we denote $\nabla_X s$ the contraction of $\nabla s$ with $X$. It is a basic fact that the value of $\nabla_X s$ at $p \in M$
only depends on $s$ and $X_p$. 

Given a smooth metric $g$ in $E$, \ie{} a smooth section of $(E^*)^{\otimes2}$ such that $g_p$ is an inner product in $E_p$ for all $p \in M$,
a connection $\nabla$ is \emph{compatible with $g$} (or \emph{preserves $g$}) if $\nabla g = 0$, where we abusively still denote by $\nabla$ the induced connection in $(E^*)^{\otimes2}$. Concretely, this means that $X \cdot g(s_1, s_2) = g(\nabla_X s_1, s_2) + g(s_1, \nabla_X s_2)$.

When $E = \upT M$, $\nabla$ is called \emph{torsion-free} if $\nabla_X Y - \nabla_Y X = [X,Y]$.
A smooth metric in $\upT M$ is known as a \emph{Riemannian metric} on $M$, and one can show that there exists a unique
torsion-free connection $\nabla$ compatible with $g$, called the \emph{Levi--Civita connection}. This result is known as the \emph{fundamental theorem of Riemannian geometry}
and can be shown via the \emph{Koszul formula}, see \cite{MR3887684} or \cite{MR2088027}.

\begin{remark}
 We will discuss curvature tensors in \autoref{subsec:CurvatureTensors}.
\end{remark}

\subsection{\texorpdfstring{$f$ is harmonic if and only if $\Delta f = 0$}{f is harmonic if and only if Delta f = 0}} \label{subsec:HessianAndLaplacian}

Let $M$ and $N$ be Riemannian manifolds and $f \colon M \to N$ a smooth map. We indicate by $\nabla^M$ and $\nabla^N$ are the Levi--Civita connections of $M$ and $N$ respectively. Consider the pullback vector bundle
\begin{equation}
 E \coloneqq f^*(\upT N) \to M
\end{equation}
whose fiber above $x \in M$ is $\upT_{f(x)} N$. $E$ admits a pullback connection $f^* \nabla^N$ uniquely determined by the relation
$(f^* \nabla^N)_X (f^* s) = f^* \big(\nabla_X^N s\big)$ for any section $s \in \Gamma(\upT N)$. Henceforth we denote it $\nabla \coloneqq f^* \nabla^N$.

Consider the vector bundle $\upT^* M \otimes E \to M$. It admits a tensor product connection $\bar{\nabla}$ induced from the dual connection of $\nabla^M$ in $\upT^* M$
and the connection $\nabla$ in $E$. Thus $\bar{\nabla}$ is a linear map:
\begin{equation}
 \bar{\nabla} \colon \Gamma(\upT^* M \otimes E) \to \Gamma(\upT^* M \otimes \upT^* M \otimes E)\,.
\end{equation}

Note that $\upd f$, which is a linear map $\upT M \to \upT N$, can be seen as a section of $\upT^* M \otimes E$.
It therefore makes sense to consider $\bar{\nabla} (\upd f)$, it is an element of $\Gamma(\upT^* M \otimes \upT^* M \otimes E)$. In other words it is a bilinear map
$\upT M \times \upT M \to E$.
By definition, this is the \emph{Hessian} of $f$. Abusing notations, we denote it $\nabla^2 f \coloneqq \bar{\nabla} (\upd f)$. 
% The justification of this notation is that it is reasonable
% to write $\nabla f$ instead of $\upd f$ (for any connection $\nabla$ on $M$, $\upd f$ and $\nabla f$ coincide for real-valued functions), and if we also abusively write $\nabla$ 
% instead of $\bar{\nabla}$ for the connection in $\upT^* M \otimes E$, we indeed have $\nabla^2 f = \nabla \nabla f = \bar{\nabla}(\upd f)$.

\begin{proposition} \label{prop:HessianSymmetric}
 The Hessian $\nabla^2 f$ is symmetric as a bilinear map  $\upT M \times \upT M \to \upT N$.
\end{proposition}

\begin{proof}
This is essentially due to the fact that $\nabla^M$ and $\nabla^N$ are torsion-free. Let $u, v$ be tangent vectors at some $x \in M$,
let us show that $\nabla^2 f (u,v) - \nabla^2 f(v,u) = 0$. By definition of the product connection $\bar{\nabla}$,
\begin{equation} \label{eq:HessianSymmetricProof1}
\nabla_u(\upd f(v)) = (\bar{\nabla}_u (\upd f))(v) + \upd f (\nabla^M_u v)\,.
\end{equation}
Note that for \eqref{eq:HessianSymmetricProof1} to make sense, we locally extend $v$ as a vector field
around $x$.
We rewrite \eqref{eq:HessianSymmetricProof1} as:
\begin{equation} 
\nabla^2 f(u,v) = \nabla^N_{\upd f(u)}(\upd f(v)) - \upd f(\nabla^M_u v)\,.
\end{equation}
Thus we can write
\begin{equation} \label{eq:HessianSymmetricProof2}
\nabla^2 f(u,v) - \nabla^2 f(v,u)= \left[\nabla^N_{\upd f(u)}(\upd f(v)) - \nabla^N_{\upd f(v)}(\upd f(u))\right]- \left[\upd f(\nabla^M_u v) - \upd f(\nabla^M_v u)\right]\,.
\end{equation}
Since $\nabla^M$ and $\nabla^N$ are both torsion-free, \eqref{eq:HessianSymmetricProof2} is rewritten
\begin{equation}
\nabla^2 f(u,v) - \nabla^2 f(v,u)= \left[\upd f(u), \upd f(v)\right]- \upd f([u, v])\,.
\end{equation}
This is zero by naturality of the Lie backet.
\end{proof}

By definition, the \emph{Laplacian} $\Delta f$  is the trace of the Hessian $\nabla^2 f$. Recall that in the presence of an inner product,
one can take the trace of any symmetric bilinear form: it is the trace of the associated self-adjoint endomorphism. 
% Alternatively, it is the trace of the matrix associated to the bilinear form in any orthonormal basis. 
Let us record these definitions:

\begin{definition} \label{def:HessianAndLaplacian}
 The \emph{Hessian} of $f$ is the symmetric bilinear map $\nabla^2 f \colon \upT M \times \upT M \to \upT N$
defined by $\nabla^2 f = \bar{\nabla}(\upd f)$. The \emph{Laplacian} $\Delta f$ is
the section of  $f^*(\upT N)$ defined by $\Delta f = \tr (\nabla^2 f)$.
\end{definition}

\begin{remark}
 The Hessian $\nabla^2 f$ generalizes both the Riemannian Hessian (for real-valued functions) and the second fundamental form (for isometric
 immersions). The Laplacian $\Delta f$ is called
 \emph{tension field} by most authors and denoted $\tau(f)$, probably by herd behavior after Eells--Sampson \cite{MR0164306}.
\end{remark}

It is perfectly fine to define a harmonic map by $\Delta f = 0$: the reader
may take this definition home and skip the remainder of this subsection. For completeness, we recall below why $\Delta f = 0$ is equivalent to $f$ being
a critical point of the energy functional. Essentially, this is because $\grad \mathbf{E}(f) = -\Delta f$ (\autoref{rem:GradientEnergy}). 

Let $f \colon M \to N$ be a smooth map. If $M$ is not compact, the energy $\mathbf{E}(f)$ can be infinite, so instead we define it on compact subsets $K \subseteq M$:
\begin{equation} \label{eq:EnergyFunctional}
 \mathbf{E}_K(f) \coloneqq \frac{1}{2} \int_K \Vert \upd f \Vert^2 \upd v_M\,.
\end{equation}
In \eqref{eq:EnergyFunctional}, $\Vert \upd f \Vert$ is the \emph{Hilbert--Schmidt norm} of $f$: Given any linear map $L \colon V \to W$ 
between Euclidean vector spaces $(V, g_V)$ and $(W, g_W)$, the Hilbert-Schmidt norm
is defined by $\Vert L \Vert^2 = \tr_{g_V}(L^* g_W)$. It is the norm relative to the tensor inner product $g_V^* \otimes g_W$. 
% More concretely, $\Vert L \Vert^2 = \tr(\transpose{A} A)$ where $A$ is the matrix of $L$ in any orthonormal bases of $V$ and $W$.

Now let $V$ be any infinitesimal variation of $f$, \ie{} $V \in \Gamma(f^* \upT N)$, and let $(f_t)$ be any $1$-parameter variation of $f$ with initial tangent vector $V$, \ie{}
$(f_t) \colon I \times M \to N$ is smooth where $I \subseteq \R$ is an interval containing $0$, $f_0 = f$, and $\ddt_{|t=0} f_t = V$. 
This variation is called supported in $K$ if $f_t = f$ outside of $K$.

\begin{theorem}[First variational formula for the energy] \label{prop:FirstVariationalFormula}
For any variation $(f_t)$ of $f$ supported in a compact set $K$ and with initial tangent vector $V$,
 \begin{equation} \label{eq:FirstVariationalFormula}
  \ddt_{|t=0} \mathbf{E}_K(f_t) = - \int_K \langle \Delta f, V \rangle_N \upd v_M\,.
 \end{equation}
\end{theorem}

\begin{proof}
The first variational formula for the energy is essentially an integration by parts. First write:
%  \begin{equation}
%  \begin{aligned}
%    \ddt_{|t=0} \mathbf{E}_K(f_t) &= \frac{1}{2} \int_K  \ddt_{|t=0} \tr \langle \upd f_t, \upd f_t \rangle_N \upd v_M \\
%    &= \int_K  \tr \langle ( \upd f, \nabla_{\frac{\partial}{\partial t}} \upd f_t)_{|t=0} \rangle_N \upd v_M 
%  \end{aligned}
% \end{equation}
 \begin{equation}
   \ddt_{|t=0} \mathbf{E}_K(f_t) = \frac{1}{2} \int_K  \ddt_{|t=0} \tr \langle \upd f_t, \upd f_t \rangle_N \upd v_M
   = \int_K  \tr \langle ( \upd f, \nabla_{\frac{\partial}{\partial t}} \upd f_t)_{|t=0} \rangle_N \upd v_M 
\end{equation}
where $\nabla$ here denotes the product connection in $\R \times M$. For this connection, one has 
$\nabla_\frac{\partial}{\partial t} \nabla_u =  \nabla_u \nabla_\frac{\partial}{\partial t}$
for any $u \in \upT M$, from this observation one can derive that 
$ \left( \nabla_{\frac{\partial}{\partial t}} \upd f_t\right)_{|t=0} = \nabla V$.
We therefore get
 \begin{equation} \label{eq:proofFVF1}
   \ddt_{|t=0} \mathbf{E}_K(f_t)
   = \int_K  \tr \langle \upd f, \nabla V \rangle_N \upd v_M \,.
\end{equation}
By compatibility of $\nabla$ with the metric in $N$, the function $\tr \langle \upd f, \nabla V \rangle_N$ can be written as
$-\langle \tr  \nabla^2 f, V \rangle_N$ plus a coexact function (namely $\upd^* \langle \upd f (\cdot), V \rangle$), and is zero outside of $K$. Stokes's theorem thus yields:
%  \begin{equation}
%  \begin{aligned}
%    \ddt_{|t=0} \mathbf{E}_K(f_t)
%    &= -\int_K  \langle \tr  \nabla^2 f,  V \rangle_N \upd v_M \\
%    &= -\int_K  \langle \Delta f, V \rangle_N \upd v_M \,.
%  \end{aligned}
% \end{equation}
 \begin{equation}
   \ddt_{|t=0} \mathbf{E}_K(f_t)
   = -\int_K  \langle \tr  \nabla^2 f,  V \rangle_N \upd v_M
   = -\int_K  \langle \Delta f, V \rangle_N \upd v_M \,.
\end{equation}
\end{proof}

\begin{remark} \label{rem:gradE}
The end of the proof can be rewritten more convincingly using the tools of \autoref{subsec:dfHarmonic}:
\begin{equation}
 \begin{split}
  \ddt_{|t=0} \mathbf{E}_K(f_t)
  & = \int_K  \tr \langle \upd f, \nabla V \rangle_N \upd v_M \\
  & = \langle \upd f, \nabla V \rangle_{\upL^2} = \langle \upd^*_\nabla \upd f, V \rangle_{\upL^2} = - \langle \Delta f, V \rangle_{\upL^2}\,.
 \end{split}
\end{equation}
\end{remark}

\begin{remark} \label{rem:GradientEnergy}
The first variational formula for the energy \eqref{eq:FirstVariationalFormula} says precisely that 
$\grad \mathbf{E}(f) = - \Delta f$
on $\cC^\infty(M,N)$, where the gradient is taken with respect to the Riemannian metric defined by 
$\langle U, V \rangle = \int_M \langle U, V \rangle_N \upd v_M$ for all $U, V \in \Gamma(f^* \upT N)$.
Even though $\cC^\infty(M,N)$ is infinite-dimensional, it can be equipped with a smooth structure making the previous statement precise
(see \cite[Chap. 42]{MR1471480}).
\end{remark}

One says that $f$ is a critical point of the energy functional if $\ddt_{|t=0} \mathbf{E}_K(f_t) = 0$ for any compact $K \subseteq M$ and for any variation $(f_t)$ of $f$ supported in $K$.
According to the previous remark, this can be simply put: $\grad \mathbf{E}(f) = 0$. In any case, the next corollary follows immediately from \autoref{prop:FirstVariationalFormula}:
\begin{corollary}
 Let $f \colon M \to N$ be a smooth map between Riemannian manifolds.
 \begin{equation}
  \text{$f$ is harmonic} \quad \stackrel{\text{def}}{\Leftrightarrow} \quad \text{$f$ is a critical point of the energy functional}  \quad \Leftrightarrow \quad \Delta f = 0\,.
 \end{equation}
\end{corollary}

\subsection{\texorpdfstring{$f$ is harmonic if and only if $\upd_\nabla^* \upd f = 0$}{f is harmonic if and only if d*df = 0}}
\label{subsec:dfHarmonic}

We have seen in \autoref{subsec:HessianAndLaplacian} that $\upd f$ is a section of $\upT^*M \otimes E$ where $E = f^* \upT N$. We then defined a connection $\bar{\nabla}$
in $\upT^*M \otimes E$ and defined the Hessian of $f$ as $\bar{\nabla}(\upd f)$. Alternatively, one could see $\upd f$ as a $1$-form with values
in $E$. More generally, we denote $\Omega^k(M, E)$ the space of smooth $k$-forms with values in $E$:
\begin{equation}
 \Omega^k(M, E) \coloneqq \Gamma(\Lambda^k \upT^* M \otimes E)\,.
\end{equation}
The connection $\nabla$ in $E$ extends uniquely to a linear map
\begin{equation}
 \upd_\nabla \colon \Omega^k(M, E) \to \Omega^{k+1}(M, E)
\end{equation}
called the \emph{exterior covariant derivative} such that $\upd_\nabla (\omega \otimes s) = \upd \omega \otimes s + (-1)^k \omega \wedge \nabla s$.
% \end{equation}
% \begin{equation} \label{eq:ExteriorCovariantDerivative}
%  \upd_\nabla (\omega \otimes s) = \upd \omega \otimes s + (-1)^k \omega \wedge \nabla s
% \end{equation}
% for all $\omega \in \Omega^k(M, \R)$ and $s \in \Gamma(E)$.

Note that $\upd_\nabla$ does not see the metric on $M$: it only depends on $\nabla = f^* \nabla^N$.
In particular, one can consider the $2$-form $d_\nabla(\upd f) \in \Omega^2(M, E)$, which only depends on the metric on $N$, but we shall soon see that $d_\nabla(\upd f) = 0$.

On the other hand, there is a tensor product connection $\bar{\nabla} \colon \Gamma(\upT^* M \otimes E) \to \Gamma(\upT^* M \otimes \upT^* M \otimes E)$
and more generally a tensor product connection  
\begin{equation}
 \bar{\nabla} \colon \Gamma((T^* M)^{\otimes k} \otimes E) \to \Gamma((T^* M)^{\otimes (k+1)} \otimes E)\,.
\end{equation}
Since $\Lambda^k \upT^* M$ is a subspace of $(T^* M)^{\otimes k}$ (namely the subspace of antisymmetric tensors), one can restrict $\bar{\nabla}$ to this subspace and get
a map $\bar{\nabla} \colon  \Omega^k(M, E) \to \Gamma((T^* M)^{\otimes (k+1)} \otimes E)$. 
% It turns out that $\upd_\nabla$ is the antisymmetrization of $\bar{\nabla}$:

\begin{proposition} \label{prop:Antisymmetrization}
 $\upd_\nabla \colon \Omega^k(M, E) \to \Omega^{k+1}(M, E)$ is the antisymmetrization of the restriction of $\bar{\nabla}$ to $\Omega^k(M, E)$.
 Concretely, given $\alpha \in \Omega^k(M, E)$:
 \begin{equation} \label{eq:antisymmetrization}
  \upd_\nabla \alpha (u_0, \dots, u_k) = \sum_{s=0}^k (-1)^s \, (\bar{\nabla}_{u_s} \alpha)(u_0, \dots, \widehat{u_s}, \dots, u_k)
 \end{equation}
 where the notation $\widehat{u_s}$ means that $u_s$ is omitted.  For example ($k=1$):
 \begin{equation} \label{eq:antisymmetrization1}
  (\upd_\nabla \alpha)(u,v) = (\bar{\nabla} \alpha)(u,v) - (\bar{\nabla} \alpha)(v,u)
 \end{equation}
\end{proposition}
\begin{proof}
 It suffices that the antisymmetrization of $\bar{\nabla}$ verifies the characterization of the exterior covariant derivative,
 which is readily checked.
\end{proof}

% Recall that $\upd f \in \Omega^1(M,E)$, so that $\upd_\nabla (\upd f) \in \Omega^2(M,E)$.
\begin{proposition} \label{prop:dfclosed}
 For any smooth $f \colon M \to N$, $\upd_\nabla(\upd f) = 0$.
\end{proposition}
\begin{proof}
 By \autoref{prop:Antisymmetrization}, $\upd_\nabla(\upd f)$ is the antisymmetrization of $\bar{\nabla}(\upd f)$. But recall that the Hessian 
 $\nabla^2 f \coloneqq \bar{\nabla}(\upd f)$ is symmetric: \autoref{prop:HessianSymmetric}.
\end{proof}

\autoref{prop:dfclosed} says that $\upd f$ is always a closed $1$-form. Let us show that $f$ is harmonic if and only if $\upd f$ is co-closed. First we need to introduce the Hodge star and the codifferential.

\begin{definition} \label{def:battery}
Let $M$ be a Riemannian manifold and let $E \to M$ be a smooth vector bundle with a metric $\langle \cdot , \cdot \rangle_E$.
\begin{itemize}
 \item  The \emph{mixed product} of $E$-valued differential forms is the operation:
 \begin{equation}
  \begin{split}
   \Omega^k(M,E) \times \Omega^l(M,E) &\to \Omega^{k+l}(M, \R)\\
   (\alpha, \beta) & \mapsto \langle \alpha \wedge \beta \rangle
  \end{split}
 \end{equation}
defined by $\langle \omega_1 \otimes s_1 \wedge \omega_2 \otimes s_2 \rangle = \omega_1 \wedge \omega_2 \, \langle s_1, s_2 \rangle_E$.
\item The \emph{pointwise inner product} on $\Omega^k(M, E)$ is the operation: 
 \begin{equation}
  \begin{split}
   \Omega^k(M,E) \times \Omega^k(M,E) &\to \cC^{\infty}(M, \R)\\
   (\alpha, \beta) & \mapsto \langle \alpha , \beta \rangle
  \end{split}
 \end{equation}
defined by $\langle \omega_1 \otimes s_1 \wedge \omega_2 \otimes s_2 \rangle = \langle \omega_1 , \omega_2 \rangle_M \, \langle s_1, s_2 \rangle_E$, where
$\langle \cdot , \cdot \rangle_M$ is the inner product in $\Lambda^k \upT_x^* M$ induced from the inner product $\langle \cdot , \cdot \rangle_M$ in $\upT_x M$.
\item The \emph{inner product} in $\Omega^k(M, E)$ is the operation: 
 \begin{equation}
  \begin{split}
   \Omega^k(M,E) \times \Omega^k(M,E) &\to \R\\
   (\alpha, \beta) & \mapsto \langle \alpha , \beta \rangle_{\upL^2} \coloneqq \int_M \langle \alpha , \beta \rangle \upd v_M\,.
  \end{split}
 \end{equation}
 For this definition we assume that $M$ is compact or $\alpha$ and $\beta$ both have compact support.
 \item The \emph{Hodge star} in $\Omega^\bullet(M, E)$ is the operation: 
 \begin{equation}
  \begin{split}
   \Omega^k(M,E) & \to\Omega^{m-k}(M,E) \quad (m = \dim M)\\
   \beta & \mapsto \ast \beta
  \end{split}
 \end{equation}
characterized by the identity:
\begin{equation}
 \langle \alpha \wedge \ast \beta \rangle = \langle \alpha , \beta \rangle \vol_M\,.
\end{equation}
\end{itemize}
\end{definition}

The following proposition is elementary and its proof is left to the reader:
\begin{proposition}
 The Hodge star $\ast$ in $\Omega^\bullet(M, E)$ (cf \autoref{def:battery}) is well-defined. Moreover: 
 \begin{enumerate}[(i)]
  \item For any $\alpha = \omega \otimes s \in \Omega^k(M, E)$, $\ast \alpha = (\ast \omega) \otimes s$, where $\ast \omega$ is the standard Hodge star for real-valued differential forms (\ie{} take \autoref{def:battery} with $E = \R$).
  \item The Hodge star is a pointwise linear isometry: $\langle \ast \alpha , \ast \beta \rangle = \langle \alpha, \beta \rangle$.
  \item The Hodge star is an involution up to sign: for all $\alpha \in \Omega^k(M, R)$, $\ast \ast \alpha = (-1)^{k(m-k)} \alpha$.
 \end{enumerate}
\end{proposition}

We are now ready to define the codifferential and Hodge Laplacian:
\begin{definition} \label{def:battery2}
Let $M$ be a Riemannian manifold, let $E \to M$ be a vector bundle with a metric $\langle \cdot , \cdot \rangle_E$
and with a connection $\nabla$ preserving the metric.
\begin{itemize}
 \item The \emph{codifferential} in $\Omega^\bullet(M,E)$ is the operation:
 \begin{equation}
  \begin{split}
   \upd^*_\nabla \colon \Omega^k(M, E) &\to \Omega^{k-1}(M,E)\\
   \alpha & \mapsto \upd^*_\nabla \alpha \coloneqq (-1)^{m(k-1)+1} \ast \upd_\nabla \ast \alpha\,.
  \end{split}
 \end{equation}
  \item The \emph{Hodge Laplacian} in $\Omega^k(M,E)$ is the operation:
 \begin{equation}
  \begin{split}
   \Delta \colon \Omega^k(M, E) &\to \Omega^{k}(M,E)\\
   \alpha & \mapsto \Delta \alpha \coloneqq \upd^*_\nabla \upd_\nabla  \alpha + \upd_\nabla \upd^*_\nabla  \alpha\,.
  \end{split}
 \end{equation}
A $k$-form $\alpha \in \Omega^k(M,E)$ is called \emph{harmonic} if $\Delta \alpha = 0$.
\end{itemize}
\end{definition}

The next proposition is both elementary and crucial:
\begin{proposition} \label{prop:Hodge}
 Let $M$ be a Riemannian manifold, let $E \to M$ be a vector bundle with a metric $\langle \cdot , \cdot \rangle_E$
and with a connection $\nabla$ preserving the metric.
\begin{enumerate}[(i)]
 \item \label{propHodgei} The codifferential $\upd^*_\nabla$ is the formal adjoint of the differential $\upd_\nabla$:
 \begin{equation}
  \langle \upd_\nabla \alpha, \beta \rangle_{\upL^2} = \langle \alpha, \upd^*_\nabla \beta \rangle_{\upL^2}
 \end{equation}
 whenever this is well-defined ($\deg \beta = \deg \alpha + 1$ and $\alpha$ or $\beta$ has compact support).
 \item \label{propHodgeii} A differential form $\alpha$ with compact support is harmonic if and only if it is closed and co-closed:
 \begin{equation}
  \Delta \alpha = 0 \quad \Leftrightarrow \quad \upd_\nabla \alpha = 0 ~~ \text{and} ~~ \upd^*_\nabla \alpha = 0\,.
 \end{equation}
\end{enumerate}
\end{proposition}

\begin{proof}
 For \ref{propHodgei} we write, given $\alpha \in \Omega^k(M, E)$ and $\beta \in \Omega^{k+1}(M, E)$, both with compact support:
 \begin{equation}
  \begin{split}
   \langle \alpha, \upd^*_\nabla \beta \rangle \vol_M &= \left\langle \alpha \wedge \ast(\upd^*_\nabla \beta) \right\rangle \\
   &= \left\langle \alpha \wedge \ast \left[(-1)^{mk+1} \ast \upd_\nabla \ast \beta\right] \right\rangle\\
   &= \left\langle \alpha \wedge (-1)^{mk+1} (-1)^{(m-k)k} \upd_\nabla \ast \beta \right\rangle\\
   &= (-1)^{k+1} \left\langle \alpha \wedge  \upd_\nabla \ast \beta\right\rangle \,.\\
  \end{split}
 \end{equation}
 Now write
 \begin{equation}
  \begin{split}
   \upd \langle \alpha \wedge \ast \beta \rangle &= \langle  \upd_\nabla \alpha \wedge \ast \beta \rangle + (-1)^k \langle \alpha \wedge \upd_\nabla \ast \beta \rangle\\
   &= \langle \upd_\nabla \alpha , \beta\rangle \vol_M -  \langle \alpha, \upd^*_\nabla \beta \rangle \vol_M
  \end{split}
 \end{equation}
and integrate over $M$ (use Stokes's theorem) to find \ref{propHodgei}. 
Now \ref{propHodgeii} follows easily:
 \begin{equation}
 \begin{split}
  \langle \Delta \alpha, \alpha \rangle_{\upL^2}
  &= \langle \upd^*_\nabla \upd_\nabla \alpha, \alpha \rangle_{\upL^2} + \langle  \upd_\nabla \upd^*_\nabla \alpha,  \alpha \rangle_{\upL^2}\\
  &= \langle \upd_\nabla \alpha, \upd_\nabla \alpha \rangle_{\upL^2} + \langle \upd^*_\nabla \alpha, \upd^*_\nabla \alpha \rangle_{\upL^2}\\
  \end{split}
 \end{equation}
 and observe that $\langle \upd_\nabla \alpha, \upd_\nabla \alpha \rangle \geqslant 0$ with equality if and only if $\upd_\nabla \alpha = 0$, same for
 $\langle \upd^*_\nabla \alpha, \upd^*_\nabla \alpha \rangle$.
\end{proof}

Out of interest, let us mention the main theorem of Hodge theory in the classical case $E = \R$:
\begin{theorem}
 Let $M$ be a closed Riemannian manifold. There is an orthogonal decomposition
 \begin{equation}
  \Omega^k(M, \R) 
  = \lefteqn{\underbrace{\phantom{\cH^k(M, \R)\oplus \operatorname{Im}(\upd)}}_{\ker (d)}}\overbrace{\cH^k(M, \R)}^{\ker \Delta}\oplus \overbrace{\operatorname{Im}(\upd)\oplus\operatorname{Im}(\upd^*)}^{\Imag \Delta}\,.
 \end{equation}
Moreover, the space of harmonic $k$-forms $\cH^k(M, \R) \coloneqq \ker \Delta$  is finite-dimensional.
\end{theorem}

\begin{proof}[Proof sketch]
The proof relies on the \emph{finiteness theorem}, which says that if $P$ is an elliptic operator in a vector bundle $E$, 
then $\ker P$ is finite-dimensional and there is an orthogonal decomposition $\Gamma(E) = \ker P \oplus \Imag P^*$ (see \cite[Thm.\ 3.10]{MR1924513}).
Since $\Delta$ is elliptic and self-adjoint, we obtain $\Omega^k(M, \R) = \ker \Delta \oplus \operatorname{Im} \Delta$ as expected. 
To show that $\Imag \Delta = \operatorname{Im}(\upd) \oplus \operatorname{Im}(\upd^*)$, one argues that $\operatorname{Im}(\upd)$ and $\operatorname{Im}(\upd^*)$ are orthogonal
(because $\upd \circ \upd = 0$); they are also orthogonal to $\ker \Delta$ therefore contained in $\Imag \Delta$, on the other hand it is trivial that 
$\Imag \Delta \subseteq \operatorname{Im}(\upd) + \operatorname{Im}(\upd^*)$.
Finally, $\cH^k(M, \R)\oplus \operatorname{Im}(\upd) = \ker (\upd)$ since the inclusion $\subseteq$ is clear on the one hand, and $\ker (\upd)$ is orthogonal to $\Imag(\upd^*)$ on the other.
\end{proof}

\begin{corollary}
The de Rham cohomology space $\operatorname{H}_{\text{dR}}^k(M, \R)$ is isomorphic to $\cH^k(M,\R)$.
\end{corollary}

\begin{remark}
The Hodge decomposition and the isomorphism $H_{\text{dR}}^k \approx \cH^k$ generalize to forms with values in a vector bundle $E$ with a metric
and a \emph{flat} connection $\nabla$ preserving the metric: the proof is the same. Note that if $\nabla$ is not flat, the de Rham cohomology is not even well-defined\footnote{In general, 
$d_\nabla^2 =  \cdot  \wedge F^\nabla$ on $\Omega^k(M, E)$, where $F^\nabla$ is the curvature of $\nabla$ (see \autoref{subsec:CurvatureTensors}). In particular, $F^\nabla = 0$ 
if and only if the de Rham complex $\left(\Omega^{\bigcdot}(M,E), \upd_\nabla\right)$ is indeed a complex \ie{} $\upd_\nabla^2 = 0$,
which is necessary and sufficient to define its cohomology.}.
\end{remark}

Now let us come back to the setting where $E = f^* \upT N$ and $\nabla = f^* \nabla^N$. We saw that $\upd f \in \Omega^1(M, E)$ is always a closed $1$-form (\autoref{prop:dfclosed}).
Weighing in \autoref{prop:Hodge} \ref{propHodgeii}, we find:
\begin{proposition} If $M$ is compact,  $\upd f \in \Omega^1(M, E)$ is harmonic if and only if $\upd^*_\nabla \upd f = 0$.
%  \begin{equation}
%   \upd f \in \Omega^1(M, E) \text{ is harmonic} \quad \Leftrightarrow \quad \upd^*_\nabla \upd f = 0\,.
%  \end{equation}
\end{proposition}

Note that $\upd^*_\nabla \upd f$ is an element of $\Omega^0(E)$, \ie{} a section of $f^*(\upT N)$, just like $\Delta f$.
\begin{proposition} \label{prop:deltanabladf}
For any smooth map $f \colon M \to N$, $  \upd^*_\nabla \upd f = -\Delta f$.
%  \begin{equation}
%   \upd^*_\nabla \upd f = -\Delta f
%  \end{equation}
\end{proposition}

\begin{proof}
 This is a special case of the formula  $\upd^*_\nabla \alpha = -\tr_{12} \bar{\nabla} \alpha$
%  \begin{equation}\label{eq:ProofDeltaNablaDf1}
%   \upd^*_\nabla \alpha = -\tr_{12} \bar{\nabla} \alpha
%  \end{equation}
for any $\alpha \in \Omega^k(M, E)$, 
 which is the $\upd^*_\nabla$-analogue of \eqref{eq:antisymmetrization}. More concretely:
  \begin{equation} \label{eq:ProofDeltaNablaDf2}
  \upd^*_\nabla \alpha(u_1, \dots, u_{k-1}) =  -\sum_{j=1}^m \bar{\nabla}_{e_j} \alpha(e_j, u_1, \dots, u_{k-1})
 \end{equation}
 where $(e_j)_{1 \leqslant j \leqslant m}$ is any local orthonormal frame field on $M$.
One can prove this formula by verifying that using it as a definition
for $\upd^*_\nabla \alpha$, it does give a formal adjoint of $\upd_\nabla$: check that $\langle \upd^*_\nabla \alpha, \beta \rangle$ is pointwise equal to 
 $\langle \alpha, \upd_\nabla \beta \rangle$ plus a globally defined co-exact function. Alternatively, a direct proof can be given using normal coordinates:
 see \cite[Lemma 1.20]{MR703510}.
\end{proof}

\begin{remark}
The formula $\upd^*_\nabla \alpha = -\tr_{12} \bar{\nabla} \alpha$ says that $\upd_\nabla^* = - \operatorname{div}$, where $\operatorname{div}$ is the divergence operator
suitably interpreted on $\Omega^k(M, E)$.
\end{remark}

\begin{remark}
\autoref{prop:deltanabladf} holds even if $M$ is not compact: the proof above shows that 
 $\langle \upd^*_\nabla \alpha +\tr_{12} \bar{\nabla} \alpha, \beta \rangle_{\upL^2} = 0$ for any \emph{compactly supported} $\beta$, but this is enough
 to conclude.
\end{remark}

\begin{remark}
 Warning! There is a sign discrepancy between the Laplacian of \autoref{def:HessianAndLaplacian} and the Hodge Laplacian:
 when $N = \R$, both operators make sense on $\Omega^0(M, \R)$, and differ by a minus sign.
 This is the well-known disagreement between the ``analyst's Laplacian'' and the ``geometer's Laplacian''.
\end{remark}

We can now wrap up:
\begin{theorem} \label{thm:CharacHarmonic}
 Let $f \colon M\to N$ be a smooth map between Riemannian manifolds.
 \begin{equation}
 \text{$f$ harmonic} \quad \Leftrightarrow \quad \Delta f = 0 \quad \Leftrightarrow \quad \upd^*_\nabla \upd f = 0  \quad \Leftrightarrow \quad\text{$\upd f$ is a harmonic $1$-form.}
\end{equation}
\end{theorem}
\begin{remark}
  $M$ must be assumed compact for the last equivalence, to apply \autoref{prop:Hodge} \ref{propHodgeii}. Otherwise it is not always true that $\Delta f = 0$ if $\Delta \upd f = 0$: take
  $M = \R^m$, $N=\R$, and $f(x_1, \dots, x_m) = x_1^2$.
\end{remark}

\section{The Bochner technique}

This section is not essential for our exposition so the reader in a hurry may skip it. We explain the Bochner and Weitzenböck formulas
and how it is typically used to produce rigidity results. The theorem of Siu and Sampson (\autoref{sec:SiuSampsonTheorem}) 
is a variation of this technique when the domain manifold is Kähler.

\subsection{Bochner formula}

The classical Bochner formula for a smooth function $f \colon M \to \R$ is
\begin{equation} \label{eq:BochnerReal}
 \frac{1}{2} \Delta \Vert \nabla f \Vert^2 - \langle \nabla f, \nabla \Delta  f \rangle = \Vert \nabla^2 f \Vert^2 + \Ric^M(\nabla f, \nabla f)\,.
\end{equation}
In this formula, $\nabla f$ is an alias for either $\upd f$ or  $\grad f$ (one can harmlessly switch using metric duality).
Eells and Sampson \cite{MR0164306} generalized the formula for a smooth function $f \colon M \to N$:
\begin{equation} \label{eq:BochnerRiemannian1}
 \frac{1}{2} \Delta \Vert \nabla f \Vert^2 - \langle \nabla f, \nabla \Delta  f \rangle = \Vert \nabla^2 f \Vert^2 + \Ric^M(\langle \nabla f, \nabla f \rangle_N) - 
 R^N(\langle \nabla f \otimes \nabla f \rangle_M, \langle \nabla f \otimes \nabla f \rangle_M)\,.
\end{equation}

Let us clarify the notations:
\begin{itemize}
 \item $\langle \nabla f, \nabla  \Delta  f \rangle$ is the pointwise inner product in $\Omega^1(M, E)$.
 \item $\langle \nabla f, \nabla f \rangle_N$ is the section of $\upT M \otimes \upT M$ obtained from pairing $\nabla f$ to itself using $\langle \cdot, \cdot \rangle_N$.
 \item $\langle \nabla f \otimes \nabla f \rangle_M$ is the section of $\upT N \otimes \upT N$ obtained from tensoring $\nabla f$ with itself using $\langle \cdot, \cdot \rangle_M$.
 \item $R^N$ is the curvature operator in $N$: see \autoref{subsec:CurvatureTensors}.
\end{itemize}

\medskip \noindent
We shall prove the Bochner formula in \autoref{subsec:Weitzenbock} via the Weitzenböck formula.

\subsection{Application of the Bochner formula}

Before clarifying curvature tensors (\autoref{subsec:CurvatureTensors}) and proving the Weitzenböck
and Bochner formulas (\autoref{subsec:Weitzenbock}), let us explain the importance of the Bochner formula for the study of harmonic maps.

\subsubsection{Rigidity}

The key idea is to require the appropriate curvature assumptions so that one controls the sign of all terms on the right-hand side of the
Bochner formula \eqref{eq:BochnerRiemannian1}. Ideally: $M$ has nonnegative Ricci curvature and $N$ has nonpositive sectional curvature.
In that case, all these terms are pointwise nonnegative.
Assuming $M$ is compact and $f$ is harmonic, one then integrates \eqref{eq:BochnerRiemannian1} to obtain rigidity results.
Indeed, we have $\Delta f = 0$ when $f$ is harmonic, and $\Delta \Vert \nabla f \Vert^2$ always integrates to zero. Let us recall
why:
\begin{lemma} \label{lem:LaplacianIntegral}
 Let $M$ be a compact Riemannian manifold. For any $u \in \cC^\infty(M, \R)$, $\int_M \Delta u \upd v_M = 0$.
\end{lemma}
\begin{proof}
Recall that a differential form of top degree $\omega$ has zero integral if and only if it is exact.
Using the definition of the codifferential $\upd^*$ and recalling that $\ast 1 = \vol_M$, this amounts to saying
that $u \vol_M$ has zero integral if and only if $u$ is co-exact. But $\Delta u = \upd^* \upd u$ is obviously co-exact.
\end{proof}

\begin{remark}
Alternatively, $\int_M \Delta u \upd v_M = 0$ follows from the divergence theorem, recalling that
$\Delta u = \operatorname{div}(\grad u)$.
\end{remark}

Let us summarize the previous observations:
\begin{equation} \label{eq:BochnerIntegrate}
 \underbrace{\frac{1}{2} \Delta \Vert \nabla f \Vert^2}_{\text{integral $=0$}} - 
 \underbrace{\langle \nabla f, \nabla \Delta  f \rangle}_{\text{$0$ if $f$ harmonic}} = 
 \underbrace{\Vert \nabla^2 f \Vert^2}_{\text{$\geqslant 0$ pointwise}} +
 \underbrace{\Ric^M(\langle \nabla f, \nabla f \rangle_N)}_{\text{$\geqslant 0$ pointwise}} ~
 \underbrace{- R^N(\langle \nabla f \otimes \nabla f \rangle_M, \langle \nabla f \otimes \nabla f \rangle_M)}_{\text{$\geqslant 0$ pointwise}}
\end{equation}

Integrating over $M$, under the previous curvature assumptions, if $f$ is harmonic then 
each of the three terms on the right-hand side must be identically zero. In particular $\nabla^2 f = 0$ everywhere: $f$ is totally geodesic. Furthermore, the vanishing of the curvature terms easily yields:
\begin{theorem}[Eells--Sampson's strong rigidity theorem] \label{thm:EellsSampsonRigidity}
 Let $f \colon M \to N$ be a smooth harmonic map between Riemannian manifolds. Assume $M$ is compact and has nonnegative Ricci curvature
 and $N$ has nonpositive sectional curvature. Then $f$ is totally geodesic. Moreover:
 \begin{enumerate}[(i)]
  \item If $\Ric^M$ is not identically zero, then $f$ is constant.
  \item If $N$ has negative sectional curvature, then $f$ is constant or maps to a closed geodesic.
 \end{enumerate}
\end{theorem}

\subsubsection{Heat flow}
\label{subsubsec:HeatFlow}

The Bochner formula is also key for the heat flow technique developed by Eells--Sampson 
\cite{MR0164306} and successfully to adapted to various settings by many authors. Let us quickly explain this, although
we will not directly use the heat flow in this report. 

Assume $(f_t)_{t \in I}$ is a $1$-parameter family of maps $M \to N$
satisfying the heat flow equation:
\begin{equation} \label{eq:HeatFlow}
 \partial_t f_t = \Delta f_t \,.
\end{equation}
One can show local existence of this flow given $f_0$ using classical nonlinear parabolic PDE techniques (linearization
of the operator $\partial_t - \Delta$ and implicit function theorem, see \eg{} \cite{MR756629}).
Note that the heat flow is just the gradient flow for the energy functional, since $\grad \mathbf{E}(f) = -\Delta(f)$ (\autoref{rem:GradientEnergy}).

Assume $N$ is nonpositively curved. 
The second variational formula for the energy shows that it is a convex functional (see \eg{} \cite[Prop. 3.4]{Gaster-Loustau-Monsaingeon1}),
making it reasonable to expect that the heat flow might converge to an energy minimizer. 
When $M$ and $N$ are both compact, one can try proving convergence of the flow (maybe up to subsequence)
using some compactness argument such as
the Arzelà–Ascoli theorem. However there are significant obstacles to overcome, such as proving the long-time
existence of the heat flow and equicontinuity of $(f_t)$. The Bochner formula shows that $\Vert \nabla f_t (x) \Vert$ is uniformly bounded
in time and space, solving both these obstacles.

Let us give a some details. Denote $e(f_t) \coloneqq \frac{1}{2} \Vert \upd f_t \Vert^2$ the energy density of $f_t$. If $(f_t)_{t \in I}$ satisfies \eqref{eq:HeatFlow}, then
$\partial_t e(f_t) = \langle \nabla f_t, \nabla  \Delta  f_t \rangle$. 
Assuming $\Ric^M$ is bounded below by $K \in \R$ (\eg{} $M$ is compact) and $N$ has nonpositive sectional curvature, the Bochner formula 
\eqref{eq:BochnerRiemannian1} yields:
\begin{equation}
 (\partial_t - \Delta) e(f_t) \leqslant K' e(f_t)
\end{equation}
where $K' = -2K$. By a generalized mean value property (more precisely: Moser's Harnack inequality for subsolutions of the heat equation
\cite{MR0159139}\cite[Lemma 5.3.4]{MR2431658}), this implies
\begin{equation} \label{eq:EnergyDensityIneq}
 \Vert e(f_t) \Vert_{\infty} \leqslant C \mathbf{E}(f_0)
\end{equation}
for some constant $C>0$. In other words, the family $(f_t)$ has a uniformly bounded gradient.
It obviously implies that it is equicontinuous, but also long-time existence of the heat flow 
(when $N$ is compact) by a standard ``blow up in finite time'' argument for nonlinear parabolic PDEs.
Thus one can extract $t_k \to +\infty$ such that $f_{t_k}$ converges uniformly to some map $f_\infty$. There still remains work to do, involving
regularity theory and Sobolev spaces, to show that $f_\infty$ is a smooth harmonic map, see \eg{} \cite{MR756629} for details.
Let us record the following conclusion:
\begin{theorem}[Eells--Sampson's existence theorem] \label{thm:EellsSampson}
 Let $f \colon M \to N$ be smooth, where $M$ and $N$ are compact Riemannian manifolds and $N$ has nonpositive sectional curvature.
 Then $f$ is homotopic to a smooth harmonic map.
\end{theorem}

\subsection{Curvature tensors} \label{subsec:CurvatureTensors}

Now is a good time to clarify our notations and conventions for curvature tensors.

Given a vector bundle $E \to M$ with a connection $\nabla$, it is easy to check that the linear map $\upd_\nabla^2 \colon \Omega^0(M, E) \to \Omega^2(M, E)$ 
is tensorial, \ie{} $\cC^\infty(M, \R)$-linear. Therefore there exists a tensor field $F^\nabla \in \Omega^2(\End E)$ such that $\upd_\nabla^2 (s) = \, F^\nabla s$. 
The operator $F^\nabla$ is called the \emph{curvature}
of $\nabla$. Recalling that $\upd_\nabla$ is the antisymmetrization of $\nabla$ as in \eqref{eq:antisymmetrization1}, $F^\nabla$ is concretely given  by 
\begin{equation} \label{eq:CurvatureConnection}
 F^\nabla(X,Y)s = \nabla_{X,Y}^2 s - \nabla_{Y,X}^2 s = \nabla_X \nabla_Y s - \nabla_Y \nabla_X s - \nabla_{[X,Y]} s
\end{equation}
where $\nabla$ is assumed torsion-free for the second identity. More elegantly: $F^\nabla(X,Y) = [\nabla_X,  \nabla_Y] - \nabla_{[X,Y]}$.

\subsubsection{Riemannian curvature tensors} \label{subsubsec:RiemannianCurvatureTensors}

 If $M$ is a Riemannian manifold,  the \emph{Riemann curvature tensor} 
 $R \coloneqq F^\nabla \in \Gamma(\Lambda^2 \upT^* M \otimes \End(\upT M))$ is the curvature of the Levi-Civita connection $\nabla$:
\begin{equation}
R(X,Y)Z = \nabla_X \nabla_Y Z - \nabla_Y \nabla_X Z - \nabla_{[X,Y]} Z\,.
\end{equation}
The \emph{purely covariant version} of the Riemann curvature tensor is the $4$-covariant tensor field:
\begin{equation} \label{eq:PurelyCovariantCurvatureTensor}
 R(X, Y, Z, W) = -\langle R(X,Y)Z, W \rangle
\end{equation}
and the \emph{curvature operator} is the symmetric bilinear form on $\otimes^2 \upT M$ defined by:
\begin{equation} \label{eq:CurvatoreOperator}
 Q(X \otimes Y, Z \otimes W) = R(X,Y, Z, W)
\end{equation}
for decomposable tensors, and extended bilinearly to $\medotimes^2 \upT M$ (it can also be defined as a bilinear form on $\Lambda^2 \upT M$
or as an endomorphism of $\medotimes^2 \upT M$ or $\Lambda^2 \upT M$). The \emph{sectional curvature} $K$ is
\begin{equation}
 K(X,Y) = \frac{Q(X \otimes Y, X\otimes Y)}{\Vert X \wedge Y \Vert^2}
\end{equation}
with $\Vert X \wedge Y \Vert^2 = \Vert X \Vert^2 \Vert Y \Vert^2 - \langle X, Y \rangle^2$. 
It is defined for any two linearly independent vectors $X$ and $Y$ with same basepoint and only depends on the plane spanned by $X$ and $Y$.
Finally, the \emph{Ricci curvature tensor} is the symmetric bilinear form on $\upT M$ defined by
\begin{equation}
 \Ric(X,Y) = \tr R(X, \cdot, Y, \cdot) = \sum_{j=1}^n R(X, e_j, Y, e_j)
\end{equation}
for any local orthonormal frame field $(e_j)_{1 \leqslant j \leqslant n}$.

\begin{remark}
 It is certainly questionable to define the purely covariant Riemann tensor with a minus sign as in \eqref{eq:PurelyCovariantCurvatureTensor}
 and still use the same letter $R$! The main reason for this choice is nicer formulas for sectional curvatures.
 Some authors introduce the minus sign earlier (\eg{} $R = -d_\nabla^2$), others later (\eg{} $K(X,Y) = -\frac{R(X, Y, X, Y)}{\Vert X \wedge Y \Vert^2}$),
 but no solution is entirely satisfactory. The only unanimous convention is sectional curvature: it should agree with the Gaussian curvature when $\dim M = 2$.
\end{remark}

\subsubsection{Curvature tensors in $\Omega^k(M,E)$} \label{subsubsec:CurvatureTensorsOmega}

In order to express the Weitzenböck formula, we generalize the Riemann and Ricci curvature to bundles of vector-valued forms.
Let $M$ be a Riemannian manifold, let $E \to M$ be a vector bundle with a metric $\langle \cdot , \cdot \rangle_E$
and a connection $\nabla$ preserving the metric. Recall that there is a product connection $\bar{\nabla}$ in the bundle $\Lambda^k \upT^* M \otimes E$.
By definition, the curvature of that bundle is the curvature of $\bar{\nabla}$.  Explicitly:
\begin{equation}
 (R(X,Y)\alpha)(u_1, \dots, u_k) = F^\nabla(X,Y)(\alpha(u_1, \dots, u_k)) - \sum_{j=1}^k \alpha(u_1, \dots, R^M(X,Y)u_j, \dots , u_k)\,.
\end{equation}
We then define the \emph{Ricci (or Bochner, or Weitzenböck) operator} $S \colon \Omega^k(M, E) \to \Omega^k(M,E)$:
\begin{equation}
  S(\alpha)(u_1, \dots, u_k) = \sum_{s=1}^k (-1)^{s-1} \tr \left[(R(\, \cdot\, , u_s) \alpha)(\, \cdot \,, u_1, \dots, \widehat{u_s}, \dots, u_k)\right]
\end{equation}
Note that when $E = \R$ is the trivial flat bundle and $k = 1$, $S(\alpha) = \Ric^M (\alpha)$, interpreting $\Ric^M$ as a self-adjoint endomorphism
of $\Omega^1(M, \R)$ by metric duality.

\subsection{Weitzenböck and Bochner formulas} \label{subsec:Weitzenbock}

We conclude with the neat proof of the Bochner formula via the Weitzenböck formula.

Let $M$ be a Riemannian manifold, let $E \to M$ be a vector bundle with a metric $\langle \cdot , \cdot \rangle_E$ and with a connection $\nabla$ preserving the metric. In $\Omega^k(M,E)$ we have two Laplacian operators:
\begin{itemize}
 \item The Hodge Laplacian $\Delta$ introduced in \autoref{def:battery2}.
 \item The \emph{connection Laplacian} (or \emph{trace Laplacian} or \emph{rough Laplacian}) 
 $\tr \left(\bar{\nabla}^2\right)$, which can be seens as the natural extension of the Laplacian on functions (see \autoref{def:HessianAndLaplacian}).
\end{itemize}

\begin{remark}
The connection Laplacian $\tr(\nabla^2)$ is well-defined in $\Gamma(E)$ for any vector bundle $E$ with a metric and a compatible connection $\nabla$.
% ; we specialized above to the bundle $\Lambda^k \upT^* M \otimes E$. 
Some call \emph{Bochner Laplacian} the operator $\nabla^* \nabla$ where $\nabla^*$ is the formal adjoint of $\nabla$. 
Since $\nabla^* = -\tr \nabla$, the Bochner Laplacian is simply minus the connection Laplacian. 
\end{remark}

The Weitzenböck formula relates both Laplacian operators:

\begin{theorem}[Weitzenböck formula]
 For all $\alpha \in \Omega^k(M,E)$,
 \begin{equation} \label{eq:WeitzenbockFormula}
  \Delta \alpha = -\tr (\bar{\nabla}^2 \alpha) + S(\alpha)\,.
 \end{equation}
\end{theorem}

\begin{remark}
For $k=0$, the Weitzenböck formula reduces to \autoref{prop:deltanabladf}, since $S = 0$. Be wary that the usual Laplacian (\ie{} connection Laplacian)
and the Hodge Laplacian differ by a minus sign!
\end{remark}

For $k \geqslant 1$, the Weitzenböck formula is proven by direct computation. It is written \eg{} in \cite[Prop\ 1.3.4]{MR1391729};
also refer to \cite[(1.29)]{MR703510} for this and more historical references. For readers curious about the full extent
of Weitzenböck formulas, we refer to \cite{NicolaescuMO} as a starting point.
\begin{corollary}[Bochner formula in $\Omega^k(M,E)$]
 For all $\alpha \in \Omega^k(M,E)$,
 \begin{equation} \label{eq:BochnerOmegak}
  \frac{1}{2} \Delta \Vert \alpha \Vert^2 + \langle \Delta \alpha, \alpha \rangle = \Vert \nabla \alpha \Vert^2 + \langle S(\alpha), \alpha \rangle
 \end{equation}
\end{corollary}
\begin{proof}
 \begin{equation}
  \begin{split}
   \frac{1}{2} \Delta \Vert \alpha \Vert^2 &= \frac{1}{2} \tr \nabla^2 \langle \alpha, \alpha \rangle = \tr \nabla \langle \nabla \alpha, \alpha \rangle\\
   &= \tr \langle \nabla^2 \alpha, \alpha \rangle + \tr \langle \nabla \alpha, \nabla \alpha \rangle\\
  \end{split}
 \end{equation}
 For the first term, use the Weitzenböck formula: $\tr (\bar{\nabla}^2 \alpha) = S(\alpha) - \Delta \alpha$.
 For the second term, note that $\tr \langle \nabla \alpha, \nabla \alpha \rangle$ is just $\langle \nabla \alpha, \nabla \alpha \rangle$ 
 by definition of the pointwise inner product in $\Omega^{k+1}(M, E)$. We thus get
$\frac{1}{2} \Delta \Vert \alpha \Vert^2 = \langle S(\alpha) - \Delta \alpha, \alpha \rangle + \Vert \nabla \alpha \Vert^2$.
%  \begin{equation}
%  \begin{aligned}
%    \frac{1}{2} \Delta \Vert \alpha \Vert^2 &= \langle S(\alpha) - \Delta \alpha, \alpha \rangle + \Vert \nabla \alpha \Vert^2\,.\\
%    %\frac{1}{2} \Delta \Vert \alpha \Vert^2 &+ \langle \Delta \alpha, \alpha \rangle = \Vert \nabla \alpha \Vert^2 + \langle S(\alpha), \alpha \rangle
%  \end{aligned}
%  \end{equation}
\end{proof}

\begin{corollary}[Bochner formula for maps between Riemannian manifolds]
Let $f \colon M \to N$ be a smooth map between Riemannian manifolds.
 \begin{equation} \label{eq:BochnerRiemannian2}
 \frac{1}{2} \Delta \Vert \nabla f \Vert^2 - \langle \nabla f, \nabla  \Delta f \rangle = \Vert \nabla^2 f \Vert^2 + \Ric^M(\langle \nabla f, \nabla f \rangle_N) - 
 R^N(\langle \nabla f \otimes \nabla f \rangle_M, \langle \nabla f \otimes \nabla f \rangle_M)\,.
\end{equation}
\begin{proof}
 The Bochner formula \eqref{eq:BochnerRiemannian2} is just a specialization of \eqref{eq:BochnerOmegak} when $E = f^* \upT N$, $k = 1$ and $\alpha = \upd f$.
 Indeed, $\langle S(\alpha), \alpha \rangle$ gives the two curvature terms of \eqref{eq:BochnerRiemannian2}; it
 remains to argue that $\Delta (\upd f) = - \nabla (\Delta f)$. Since $\upd_\nabla \upd f = 0$ (\autoref{prop:dfclosed}), 
 $\Delta (\upd f) = \upd_\nabla \upd_\nabla^* \upd f$. Conclude recalling that $\upd_\nabla^* \upd f = - \Delta f$ (\autoref{prop:deltanabladf}).
\end{proof}
\end{corollary}

\section{Kähler manifolds and pluriharmonic maps}

\subsection{Complex and Kähler manifolds} \label{subsec:ComplexAndKahlerManifolds}

Let $M$ be a complex manifold. In particular $M$ has an \emph{almost complex structure}, \ie{} 
$J \in \Gamma(\End \upT M)$ with $J^2 = -1$. This is essentially the scalar multiplication by $i$ in $\C^n$, pulled back on $\upT M$ by complex charts.
Here $J$ is called \emph{integrable} because it comes from a complex structure on $M$. 
The Newlander-Nirenberg theorem states that $J$ is integrable if and only if its Nijenhuis tensor vanishes, see \eg{} \cite{DemaillyBook}.
Note that $J$ induces an orientation of $M$ and we always assume this agrees with the given orientation.

The complexified tangent bundle of $\upT_\C M \coloneqq \upT M \otimes \C$ splits into $\pm i$-eigenspaces of $J$ as $\upT_\C M = \upT^{1,0} M \oplus \upT^{0,1} M$.
Accordingly, a complex tangent vector decomposes into types as $u = u^{1,0} + u^{0,1}$, with:
\begin{equation} \label{eq:TypeDecomposition}
 u^{1,0} = \frac{1}{2}\left(u - i J u \right) \qquad u^{0,1} = \frac{1}{2}\left(u + i J u \right)\,.
\end{equation}

The cotangent space $\upT^* M$ also has a complex structure still denoted $J$, defined by $J \alpha \coloneqq \alpha \circ J$, 
hence the analogous decomposition $\upT_\C^* M = {\upT^*}^{1,0} M \oplus {\upT^*}^{0,1} M$. 
Thankfully, it is true that $\alpha^{1,0}(u) = \alpha(u^{1,0})$ and $\alpha^{0,1}(u) = \alpha(u^{0,1})$, where on the right-hand side one takes the complexification of $\alpha$. 
 
If $f \colon M \to N$ is a smooth map where $M$ is a (almost) complex manifold, we write similarly
\begin{equation}
  \upd^{1,0} f = \frac{1}{2}\left(\upd f - i \upd f \circ J \right) \qquad \upd^{0,1} f = \frac{1}{2}\left(\upd f + i \upd f \circ J \right)\,.
\end{equation}
These are $\R$-linear maps $\upT M \to \upT_\C N$, or equivalently, $\C$-linear maps $\upT_\C M \to \upT_\C N$.
% In the latter description, $\upd^{1,0} f(u)$ is the same as $\upd f(u^{1,0})$, using the complex linear extension of $\upd f$.
We also denote:
\begin{equation}
 \updc f \coloneqq -\upd f \circ J = -i(\upd^{1,0} f - \upd^{0,1} f)\,.
\end{equation}

\begin{remark}
If $(N, J_N)$ is also a (almost) complex manifold---this will not be the case in most of this report---one can further decompose $\upd^{1,0} f$ and $\upd^{0,1} f$ into their $(1,0)$- and $(0,1)$-parts with respect to $J_N$. 
We write accordingly $\upd^{1,0} f = \partial f + \partial \bar{f}$ and $\upd^{0,1} f = \bar{\partial} f + \bar{\partial} \bar{f}$.
% (with this notation we have $\overline{\del f} = \delbar \bar{f}$ and $\overline{\delbar f} = \del \bar{f}$).
$f$ is called holomorphic if $\upd f \circ J_M = J_N \circ \upd f$, which amounts to $\delbar f = 0$.
Equivalently, $f$ preserves types: $\upd f (\upT^{1,0} M) \subseteq \upT^{1,0} N$.
\end{remark}

% Let us now recall the definition of Hermitian and Kähler manifolds:
\begin{definition}
 A \emph{Hermitian manifold} is a smooth manifold equipped with a Riemannian metric $g$ 
 and an integrable almost complex structure $J$ that are \emph{compatible}, meaning that $g(Ju, Jv) = g(u, v)$.
\end{definition}
Given a Hermitian manifold $(M, g, J)$, one defines the \emph{fundamental form} $\omega$ by $\omega(u, v) = g(Ju, v)$. 
It is a nondegenerate $2$-form on $M$, in fact it is a $(1,1)$-form, \ie{} $\omega(Ju, Jv) = \omega(u,v)$.
One can also define a Hermitian inner product $h$ on $(\upT M, J)$ by letting $h = g -i\omega$. 

\begin{example}
The Hermitian metric on $M = \C$ is $h = \upd z \otimes \upd \bar{z} = g - i \omega$ with $g = \upd x^2 + \upd y^2$ and $\omega = \upd x \wedge \upd y$.
This example should discourage anyone from using other conventions for the definition of $\omega$.
\end{example}

The fundamental form $\omega$ can be incoroporated in the data of a Hermitian manifold $(M, g, J, \omega)$, it being understood that $\omega(u, v) = g(Ju, v)$. 
This relation shows that any 2  of the 3 tensors $g$, $J$, $\omega$ determine the third, this matches the well-known 2-out-of-3 property of the unitary group $\U(n)$.

\begin{proposition} \label{prop:FundamentalForm}
Let $(M, g, I, \omega)$ be a Hermitian manifold.
\begin{enumerate}[(i)]
 \item \label{item:propFundamentalFormi} Denoting $n \coloneqq \dim_{\C} M$, one has
 \begin{equation}
        \frac{\omega^n}{n!} = \vol_M\,.
       \end{equation}
 \item \label{item:propFundamentalFormii} For all $\alpha \in \Omega^1(M, \R)$
       \begin{equation}
        * \alpha = \frac{\omega^{n-1}}{(n-1)!} \wedge J \alpha\,.
       \end{equation}
       (See \autoref{subsec:dfHarmonic} for the definition of the Hodge star.)
\end{enumerate}
\end{proposition}

\begin{proof}
 Both statements are linear algebra statements in $\upT_x M$. One easily proves them using an orthonormal basis
 $(e_1, J e_1, \dots, e_n, J e_n)$ of $\upT_x M$.
\end{proof}

\begin{definition}
 A Hermitian manifold $(M, g, J, \omega)$ is called \emph{Kähler} if $\upd \omega = 0$.
\end{definition}
In other words, $\omega$ is required to be a symplectic structure. The simple condition $\upd \omega = 0$
has deep consequences on the geometry of $M$, starting with the existence of holomorphic coordinates identifying $(M, g, J, \omega)$
to $\C^n$ with the flat metric to first order. We shall not develop any theory
of Kähler manifolds in this report: $\upd \omega = 0$ is all we need, especially
for the proof of the Siu--Sampson theorem following Toledo. In fact, \autoref{sec:SiuSampsonTheorem} will be a perfect illustration of 
the provocative \cite[Metatheorem 1.2]{MR1379330}: 
\begin{center}
\emph{Kähler manifolds are complex manifolds whose geometry reduces to linear algebra.}
\end{center}
The only analysis involved in the proof in \autoref{subsubsec:Toledo} is that $\int_M \eta = 0$ if $\eta$ is exact, the rest is linear algebra!

\medskip 
One can show that $\upd \omega = 0$ is equivalent to $\nabla J = 0$. A straightforward yet key consequence is:
\begin{proposition} \label{prop:CurvatureKahler}
Let $(M, g, J, \omega)$ be a Hermitian manifold. If $M$ is K\"ahler then 
the Riemann and Ricci tensors of $M$ are of type $(1,1)$, \ie{}
 $R^M(Ju, Jv) = R^M(u,v)$ and $\Ric^M(Ju, Jv) = \Ric^M(u,v)$.
\end{proposition}

\subsection{Pluriharmonic maps}
\label{subsec:Pluriharmonic}

In this subsection, $N$ is any Riemannian manifold. 
If $(M, g, J, \omega)$ is a Kähler manifold, we can nicely combine \autoref{prop:deltanabladf} with \autoref{prop:FundamentalForm} \ref{item:propFundamentalFormii} and use closedness of $\omega$  to obtain:

\begin{proposition} \label{prop:DeltaHermitian}
Let $f \colon M \to N$ be a smooth map, where $(M, g, J, \omega)$ is a Kähler manifold of complex dimension $n$. Then
\begin{equation}
 -\Delta f =  \ast \left( \frac{\omega^{n-1}}{(n-1)!} \wedge \upd_\nabla \updc f \right)\,.
\end{equation}
\end{proposition}
\begin{proof}
 \begin{equation}
  \begin{aligned}
   -\Delta f &= \upd_\nabla^* \upd f & \quad& \text{by \autoref{prop:deltanabladf}}\\
   &= (-1)^{2n(k-1)+1} \ast \upd_\nabla \ast \upd f & \quad& \text{by definition of $\upd_\nabla^*$ (cf \autoref{def:battery})}\\
   &= \ast \upd_\nabla \left(\frac{\omega^{n-1}}{(n-1)!} \wedge \updc f\right) & \quad& 
   \text{by \autoref{prop:FundamentalForm} \ref{item:propFundamentalFormii}}\\
   &= \ast \left(\frac{\upd (\omega^{n-1})}{(n-1)!} \wedge \updc f + \frac{\omega^{n-1}}{(n-1)!} \wedge \upd_\nabla  \updc f \right) & \quad& 
   \text{by definition of $\upd_\nabla$}\\
   &= \ast \left(\frac{\omega^{n-1}}{(n-1)!} \wedge \upd_\nabla  \updc f \right) & \quad& 
   \text{since $\upd \omega  = 0$ ($M$ Kähler)}\\
 \end{aligned}
 \end{equation}
\end{proof}

\begin{corollary}
 Let $f \colon M \to N$ be a smooth map, where $(M, g, J, \omega)$ is a Kähler manifold of complex dimension $n$. Then $f$ is harmonic
 if and only if $\omega^{n-1} \wedge \upd_\nabla \updc f = 0$.
\end{corollary}

An important special case is when $\dim_{\C}M = 1$, \ie $M = C$ is a Riemann surface (complex curve), 
in which case any compatible metric is Kähler. Then $f \colon C \to N$ is harmonic if and only if
$\upd_\nabla \updc f = 0$. Recall that $\upd_\nabla$ does not see the metric on $C$ (only depends on $\nabla = f^* \nabla^N$),
and neither does $\updc$. Therefore $f \colon C \to N$ being harmonic only depends on the complex structure on $C$:

\begin{corollary} \label{cor:HarmonicRS}
 Let $C$ be a complex curve. The harmonicity of a smooth map $f \colon C \to N$ does not depend on the choice
 of a compatible metric on $C$.
\end{corollary}

A consequence of \autoref{cor:HarmonicRS} is the sanity of the definition of pluriharmonicity:

\begin{definition}
 A smooth map $f \colon M \to N$ where $M$ is a complex manifold is called \emph{pluriharmonic} if for every $1$-dimensional complex submanifold $C \subseteq M$,
 the restriction $f_{|C} \colon C \to N$ is harmonic.
 \end{definition}

Pluriharmonicity can alternatively be defined by the equation $\upd_\nabla \updc f = 0$:
\begin{proposition} \label{prop:CharacterizationsPluriharmonic}
 Let $f \colon M \to N$ be a smooth map where $M$ is a complex manifold.
\begin{equation}
 \text{$f$ pluriharmonic} \quad \Leftrightarrow \quad \upd_\nabla \updc f = 0\,.
\end{equation}
We also have the further characterizations:
\begin{equation}
 \text{$f$ pluriharmonic}
 \quad \Leftrightarrow \quad \delbar_\nabla \upd^{1,0} f = 0
 \quad \Leftrightarrow \quad \del_\nabla \upd^{0,1} f = 0
 \quad \Leftrightarrow \quad  \left(\nabla^2 f\right)^{1,1} = 0\,.
\end{equation}
\end{proposition}

\begin{remark}
 We have naturally written $\upd_\nabla = \upd_\nabla^{1,0} + \upd_\nabla^{0,1} = \del_\nabla + \delbar_\nabla$.
\end{remark}

\begin{remark}
To make sense of the Hessian in the last one, choose any compatible Riemannian metric on $M$ (see \autoref{def:HessianAndLaplacian}). Also see \autoref{rem:LeviForm}. Due to this characterization, 
pluriharmonic maps have sometimes been called \emph{(1,1)-geodesic} maps.
\end{remark}

\begin{proof}[Proof of \autoref{prop:CharacterizationsPluriharmonic}]
Recall that $\upd_\nabla \upd f = 0$ (\autoref{prop:dfclosed}). Decomposing into types:
\begin{equation}
 \underbrace{\del_\nabla \upd^{1,0} f}_{{(\upd_\nabla \upd f)}^{2,0} = 0} + 
 \underbrace{\del_\nabla \upd^{0,1} f + \delbar_\nabla \upd^{1,0} f}_{{(\upd_\nabla \upd f)}^{1,1} = 0} + 
 \underbrace{\delbar_\nabla \upd^{0,1} f}_{{(\upd_\nabla \upd f)}^{0,2} = 0} = 0
\end{equation}
On the other hand:
\begin{equation}
\begin{split}
 \upd_\nabla \updc f &= -i \upd_\nabla \left(\upd^{1,0} f - \upd^{0,1} f\right)\\
 &= -i \left(\del_\nabla \upd^{1,0} f + \del_\nabla \upd^{0,1} f - \delbar_\nabla \upd^{1,0} f - \delbar_\nabla \upd^{0,1} f \right)
\end{split}
\end{equation}
We thus find
\begin{equation}
 \upd_\nabla \updc f = -2i  \del_\nabla \upd^{0,1} f = 2i \delbar_\nabla \upd^{1,0} f \,.
\end{equation}
It follows that $\upd_\nabla \updc f =0 \Leftrightarrow \delbar_\nabla \upd^{1,0} f = 0 \Leftrightarrow \del_\nabla \upd^{0,1} f = 0$.
We also learn that $\upd_\nabla \updc f$ is of type $(1,1)$. 

If $\upd_\nabla \updc f = 0$, then this equation still holds on any complex curve $C \subseteq M$, 
where it means that $f$ is harmonic, hence $f$ is pluriharmonic. Conversely, if $f$ is pluriharmonic, then $\upd_\nabla \updc f = 0$ in restriction to any complex curve. As a $2$-form of type $(1,1)$, $\upd_\nabla \updc f$ vanishes if and only if it vanishes on pairs of the form
$(u, Ju)$. Since any such pair is tangent to some complex curve $C \subseteq M$, the conclusion follows.

It remains to show that $\upd_\nabla \updc f = 0 \Leftrightarrow \left(\nabla^2 f\right)^{1,1} = 0$. By \autoref{prop:Antisymmetrization},
\begin{equation}
 (\upd_\nabla \updc f)(u,v) = \bar{\nabla}_u(\updc f(v)) - \bar{\nabla}_v(\updc f(u))
 = -\nabla^2 f(u, Jv) + \nabla^2 f(v, Ju)\,.
\end{equation}
% \begin{equation}
% \begin{split}
%  (\upd_\nabla \updc f)(u,v) &= \bar{\nabla}_u(\updc f(v)) - \bar{\nabla}_v(\updc f(u))\\
%  &= -\nabla^2 f(u, Jv) + \nabla^2 f(v, Ju)\,.
% \end{split}
% \end{equation}
It follows that
\begin{equation}
 -(\upd_\nabla \updc f)(Ju,v)
 = \nabla^2 f(Ju, Jv) + \nabla^2 f(u, v)
 = 2 \left(\nabla^2 f\right)^{1,1}(u,v)
\end{equation}
% \begin{equation}
% \begin{split}
%  -(\upd_\nabla \updc f)(Ju,v)
%  &= \nabla^2 f(Ju, Jv) + \nabla^2 f(u, v)\\
%  &= 2 \left(\nabla^2 f\right)^{1,1}(u,v)
% \end{split}
% \end{equation}
($2 T^{1,1}(u,v) = T(u,v) + T(Ju,Jv)$ holds for any $2$-covariant tensor $T$.) The conclusion follows.
\end{proof}

\begin{remark} \label{rem:LeviForm}
As it turns out, the tensor $\left(\nabla^2 f\right)^{1,1} = -\frac{1}{2}(\upd_\nabla \updc f)(Ju,v)$ does not depend on the choice of a compatible metric on $M$.
One may call it the \emph{Levi form} of $f$.
\end{remark}

Below are some properties of pluriharmonic maps. For more results, we refer to \cite{MR1082880}.
\begin{proposition} \label{prop:PropertiesPluriharmonic}
 Let $f \colon M \to N$ where $M$ is a complex manifold and $N$ is a Riemannian manifold.
 \begin{enumerate}[(i)]
  \item \label{item:propPropertiesPluriharmonici} If $f$ is pluriharmonic, the restriction of $f$ to any complex submanifold of $M$ is pluriharmonic.
  \item \label{item:propPropertiesPluriharmonicii} If $f$ is pluriharmonic, then $f$ is harmonic with respect to any compatible K\"ahler metric on $M$.
  \item \label{item:propPropertiesPluriharmoniciii} If $N$ is a Kähler manifold and $f$ is $\pm$-holomorphic\footnote{We call \emph{$\pm$-holomorphic} a map that is holomorphic or antiholomorphic.}, then $f$ is pluriharmonic.
 \end{enumerate}
\end{proposition}

\begin{remark}
 When $N = \R$, a function $f \colon M \to \R$ is pluriharmonic if and only if it is the real part of a holomorphic function.
 I am unsure what a generalization of that property could be.
\end{remark}

\begin{proof}[Proof of \autoref{prop:PropertiesPluriharmonic}]
 \ref{item:propPropertiesPluriharmonici} is trivial by definition of pluriharmonicity. \ref{item:propPropertiesPluriharmonicii} follows from 
 \autoref{prop:DeltaHermitian}. 
 For \ref{item:propPropertiesPluriharmoniciii}: if $f$ is $\pm$-holomorphic, then $\updc f = \upd f \circ J_M = \pm J_N \circ \upd f$.
 Since $N$ is Kähler, $\nabla J_N = 0$ ($\nabla J = 0$ if and only if $J$ is integrable and $\upd \omega = 0$ in any almost Hermitian manifold).
 It easily follows that $\upd_\nabla (J_N \circ \alpha) = J_N \circ \upd_\nabla (\alpha)$ for any $\alpha \in \Omega^1(M, f^* \upT N)$.
 In particular, $\upd_\nabla \updc f = \pm \upd_\nabla (J_N \circ \upd f) = \pm J_N \circ \upd_\nabla \upd f$.
 Recall that $\upd_\nabla \upd f = 0$ is always true (\autoref{prop:dfclosed}), therefore we find $\upd_\nabla \updc f = 0$, \ie{} $f$ is pluriharmonic.
\end{proof}

\begin{remark}
 Nicolas Tholozan pointed out to me that the converse of \autoref{prop:PropertiesPluriharmonic} \ref{item:propPropertiesPluriharmonicii}
 is also true locally: if $f$ is harmonic with respect to any locally defined K\"ahler metric, then $f$ is pluriharmonic.
\end{remark}

\begin{example} \label{ex:MinimalSpheres}
 As a special case of \autoref{prop:PropertiesPluriharmonic} \ref{item:propPropertiesPluriharmoniciii}, if $N$ is Kähler manifold and $M$
 is a complex submanifold, then the embedding of $M$ in $N$ is minimal (\ie{} it is a harmonic isometric immersion). For example,
 consider complex submanifolds $\CP^1 \to \CP^3$ induced by linear subspaces $\C^2 \to \C^4$. Such embeddings of $\CP^1 \approx S^2$ in $\CP^3 \approx S^4$
 thus are minimal spheres.
\end{example}

\section{The Siu--Sampson theorem} \label{sec:SiuSampsonTheorem}

The main goal of this section is to explain the Siu--Sampson theorem below  (\autoref{thm:SiuSampson}), first proven by Sampson \cite{MR833809}\footnote{See
Theorem 1 p.\ 129: Although Sampson does not use the terminology \emph{pluriharmonic}, the equations
$y_{\alpha | \bar{\beta}}^j = 0$ mean that $\upd_{\nabla}^{0,1} \upd^{1,0}f = 0$,
\ie{} $f$ is pluriharmonic by \autoref{prop:CharacterizationsPluriharmonic}.
} following 
the Bochner technique of Siu \cite{MR584075}. 
\begin{theorem}[Siu, Sampson] \label{thm:SiuSampson}
 Let $M$ be a compact Kähler manifold and let $N$ be Riemannian manifold of nonpositive Hermitian sectional curvature.
 If $f \colon M \to N$ is harmonic, then $f$ is pluriharmonic. Moreover:
 \begin{equation} \label{eq:SiuSampson}
   R^N\left(X, Y, \bar{X}, \bar{Y}\right) = 0
 \end{equation}
 for all $x\in M$ and for all $X, Y \in \upd f(\upT_x^{1,0} M)$.
\end{theorem}

In \eqref{eq:SiuSampson}, $R^N$ is the complexification of the Riemann curvature tensor of $N$ to 
$\upT_{\C} N$. By definition, $N$ has nonpositive Hermitian sectional curvature if $R^N (X, Y, \bar{X}, \bar{Y}) \leqslant 0$
for all $X, Y$.

\subsection{Strong curvature conditions} \label{subsec:HermitianSectionalCurvature}

\subsubsection{Hermitian sectional curvature and strongly nonpositive curvature} 

\autoref{thm:SiuSampson} requires the notion of Hermitian sectional curvature. Contrary to what its name could suggest,
this is a purely Riemannian notion.

Let $N$ be a Riemannian manifold. Refer to  \autoref{subsubsec:RiemannianCurvatureTensors} for 
the definitions of the curvature tensors on $N$, such as the Riemann tensor $R$ and the curvature operator $Q$.
We still denote $R$ and $Q$ the complex linear extensions of these operators to complexified vectors, \eg{} 
$Q \colon \medotimes^2 (\upT_{\C} N) \times \medotimes^2 (\upT_{\C} N) \to \C$.

We say that \emph{$N$ has nonpositive Hermitian sectional curvature}
if $R(X, Y, \bar{X}, \bar{Y}) \leqslant 0$
for any $X, Y \in \upT_{\C} N$. Following Siu \cite{MR584075}, we alternatively say that \emph{$N$ has strongly nonpositive curvature}.

Clearly, $N$ has strongly nonpositive curvature if and only if $Q(\sigma, \bar{\sigma}) \leqslant 0$ for any decomposable tensor $\sigma \in \medotimes^2 (\upT_{\C} N)$.
If moreover $Q(\sigma, \bar{\sigma}) \leqslant 0$ on all $\medotimes^2 (\upT_{\C} N)$, we say that \emph{$N$ has very strongly nonpositive curvature}. This is equivalent to 
$Q(\sigma, \sigma) \leqslant 0$ on $\medotimes^2 \upT N$, in other words $Q$ is negative semidefinite on $\medotimes^2 \upT N$.
We have obvious analogous notions of (very) strongly nonnegative curvature.

\begin{example}
 A manifold of strongly nonpositive curvature clearly has nonpositive sectional curvature.
 It is easy to see that the converse holds for a manifold of constant sectional curvature.
 More generally it holds for $\delta$-pinched curvature with $\delta \geqslant 1/4$ by \cite{MR1091617}. This means
 that for all $p\in N$, the sectional curvature at $p$ is between $-c_p$ and $-\delta c_p$ for some constant $c_p > 0$ ($\delta$ is always assumed in $(0,1]$).
\end{example}

\begin{example}
 Of course, one expects that there are Riemannian manifolds with nonpositive yet not strongly nonpositive curvature, and it is claimed in \cite{MR1379330} and \cite{MR1777835}, 
 but I am not aware of an example. That being said, \cite{MR2085169} produces examples of closed Riemannian manifolds of negative sectional curvature which do not admit any Riemannian metric
 of \emph{very} strongly nonpositive curvature.
\end{example}
\begin{example}
 By Jost--Yau \cite{MR690192}, a product of hyperbolic spaces $\H^2 \times \dots \times \H^2$ has strongly nonpositive curvature. I believe that it actually has very strongly nonpositive curvature,
 in fact I do not know an example of manifold with strongly, but not very strongly, nonpositive curvature.
\end{example}
\begin{example}
 Symmetric spaces of noncompact type have very strongly
 nonpositive curvature (see \autoref{sec:SymmetricSpaces}). 
 Similarly, symmetric spaces of compact type have very strongly nonnegative curvature.
\end{example}
\begin{example}
 Y. Wu \cite{MR3189463} showed that the Teichmüller space of a closed surface of genus $>1$ has very strongly nonpositive curvature.
\end{example}

\subsubsection{Strongly negative curvature of K\"ahler manifolds} \label{subsubsec:StronglyNegKahler}

The natural attempt to define very strongly \emph{negative} curvature by $Q(\sigma, \bar{\sigma}) < 0$ 
for all nonzero $\sigma \in \medotimes^2 (\upT_{\C} N)$ is unreasonable. Indeed,
on any Kähler manifold  $N$, it follows from \autoref{prop:CurvatureKahler} that $Q(\sigma, \bar{\sigma}) = 0$ for any $\sigma$ of type $(2,0)$ or $(0,2)$.
By definition, a Kähler manifold $N$ has \emph{strongly negative curvature} if $Q(\sigma, \bar{\sigma}) < 0$ 
for all  nonzero $\sigma \in \upT^{1,0} N \otimes \upT^{0,1} N$ of length $ \leqslant 2$, \eg{} 
$\sigma = X \otimes \bar{Y} + Z \otimes \bar{W}$, and $N$ has \emph{very strongly negative curvature} if $Q(\sigma, \bar{\sigma}) < 0$ 
for all  nonzero $\sigma \in \upT^{1,0} N \otimes \upT^{0,1} N$. 
% (In other words, $N$ has very strongly negative curvature if $Q^{1,1}$ is negative definite as a Hermitian-symmetric form on $\Lambda^{1,1} \upT N$.)

\begin{example}
 Siu  \cite{MR584075} proved that complex hyperbolic space $\C H^n$ has very strongly negative curvature. $\C P^n$ has very strongly positive curvature. Mostow--Siu \cite{MR592294} provided an example of compact K\"ahler manifold
 with very strongly negative curvature that is not a quotient of $\C H^n$.
\end{example}

\begin{example}
 According to \cite[Theorem 9.26]{MR1777835}, a Kähler manifold of complex dimension $2$ has very strong negative (resp. nonpositive) curvature
 if and only if it has negative (resp. nonpositive) sectional curvature. We expect this to be false in higher dimensions, but we are not aware of a counter-example.
\end{example}

\subsubsection{Non-example: the Penrose twistor fibration}  \label{subsubsec:PenroseFibration}

Let us illustrate with a non-example that the curvature condition is critical for the Siu--Sampson theorem \autoref{thm:SiuSampson}.
Consider the map (called \emph{Penrose twistor fibration}, see \eg{} \cite{MR791300})
\begin{equation}
 \begin{split}
  f \colon \C P^3 = \{\text{$\C$-lines in $\C^4$}\} &~\longrightarrow~ \H P^1 = \{\text{$\H$-lines in $\H^2$}\}
 \end{split}
\end{equation}
defined by $f(l) = L$, where $L$ is the unique $\H$-line in $\H^2 \approx \C^4$ containing $l$. Here $\H$ denotes the $\R$-algebra of quaternions.
Topologically, $\H P^1 \approx S^7 / S^3 \approx S^4$,
and $\H P^1$ has a natural Riemannian (in fact quaternion-Kähler) metric of constant curvature $1$, making the Hopf fibration $S^3 \to S^7 \to S^4$ a Riemannian submersion.
In particular $\H P^1$ does not have nonpositive Hermitian sectional curvature.

We claim that $f$ is harmonic. Indeed, one can check that $f$ is a Riemannian submersion, and it is known since \cite{MR0164306} that a Riemannian submersion
is harmonic if and only if its fibers are minimal submanifolds. Here the fibers
of $f$ are minimal spheres as we have seen in \autoref{ex:MinimalSpheres}.

However $f$ is not pluriharmonic. Indeed, consider a fiber $C$ of $f$ and a transverse $\C$-plane $C' \subseteq \C^4$.
For any $\varepsilon>0$, it is possible to choose $C'$ contained in the $\varepsilon$-neighborhood of $C$ in $\C P^3$ (just tilt $C$ slightly)
and such that $C'$ is not a fiber of $f$. Then the restrition of $f$ to $C'$ is nonconstant and maps $C'$ into the $\varepsilon$-neighborhood a point (namely $f(C)$) in 
$\H P^1 \approx S^4$. Such a map cannot be harmonic by the maximum principle. Since $C'$ is a complex curve in $\CP^3$, $f$ is not pluriharmonic.

For non-experts, let us explain why any harmonic map from a closed manifold $M$ to a Riemannian manifold $N$ whose image is contained in a
small ball must be constant.
For $\varepsilon$ small enough, the function $\varphi = d(p, \cdot)^2$ is convex on $B(p, \varepsilon)$, \ie{} it has positive semidefinite Hessian.
% For example, on a round sphere $S^n$, one can take an open hemisphere for $B(p, \varepsilon)$.
However, harmonic functions pull back convex functions to subharmonic functions. Indeed, one has the composition formula
\begin{equation}
\Delta (\varphi \circ f) = \upd \varphi \circ \Delta(f) + \tr \left((\nabla^2 \varphi)(\upd f, \upd f)\right)
\end{equation}
so that if $\Delta f = 0$, $\Delta (\varphi \circ f) \geqslant 0$ when $\varphi$ is convex. (The converse is true: $f$ is harmonic if it pulls back 
any locally defined convex function to a subharmonic function  \cite[Thm 3.4]{MR545705}.)
This implies that any local maximum of $\varphi \circ f$ cannot be attained at an interior point
of $M$ unless $f$ is constant by the maximum principle for subharmonic functions. By compactness of $M$, $\varphi \circ f$ does attain a global maximum, however $M$ only has interior points (no boundary), so $f$ must be constant.

\subsection{Sampson's Bochner formula} \label{subsec:BochnerKahler}

Siu \cite{MR584075} was the first to find a Bochner formula (also known as \emph{Siu's $\partial \bar{\partial}$-trick}) for maps between Kähler
manifolds and obtain rigidity results. Sampson \cite{MR833809} gave an improvement of Siu's Bochner formula that implies
\autoref{thm:SiuSampson}.

Let us explain Sampson's formula. (The reader may also refer to \cite[Eq. (4.35)]{MR1391729} for a similar exposition.)
First recall that the \emph{divergence} of a tensor field $T \in \Gamma(\upT^* M \otimes \dots)$ is given by
\begin{equation} \label{eq:divergence}
 \operatorname{div} T = \tr \nabla T
 = \sum_{j=1}^m (\nabla_{e_j} T)(e_j, \dots)
\end{equation}
% \begin{equation} \label{eq:divergence}
% \begin{split}
%  \operatorname{div} T &= \tr \nabla T\\
%  &= \sum_{j=1}^m (\nabla_{e_j} T)(e_j, \dots)
% \end{split}
% \end{equation}
where $(e_j)_{1\leqslant j \leqslant m}$ is any local orthonormal frame field. This definition makes sense for $T \in \Gamma(\upT^* M \otimes E)$
where $E$ is any vector bundle with a connection: take the tensor product connection in \eqref{eq:divergence}.

\begin{remark}
Seeing
$T$ as a $1$-form with values in $E$, \eqref{eq:ProofDeltaNablaDf2} shows that $\upd_\nabla^* T = -\operatorname{div}(T)$.
\end{remark}

Let $f \colon (M,g_M) \to (N, g_N)$ be a smooth map between Riemannian manifolds. Define the first fundamental form
$\theta \coloneqq f^* g_N$. Note that by definition, 
$\Vert \upd f \Vert^2 = \tr \theta$. Take the divergence of $\theta$:
\begin{equation}
 \operatorname{div}(\theta) = \langle \Delta f, \upd f(\cdot) \rangle + \sum_{j=1}^m \langle \upd f(e_j), (\nabla_{e_j} \upd f) (\cdot)\rangle\,.
\end{equation}
This is an element of $\Gamma(\upT^* M)$. Take the divergence again and use the Weitzenböck formula:
\begin{equation} \label{eq:BochnerDivergence}
\begin{split}
 \operatorname{div}^2(\theta) &= 2 \langle \nabla  \Delta f, \nabla f \rangle + \Vert \Delta f \Vert^2 + \Vert \nabla^2 f \Vert^2 \\
 &~ + \Ric^M(\langle \nabla f, \nabla f \rangle_N) -  R^N(\langle \nabla f \wedge \nabla f \rangle_M, \langle \nabla f \wedge \nabla f \rangle_M)\,.
 \end{split}
\end{equation}
This is merely a rewriting of the Bochner formula \eqref{eq:BochnerRiemannian2}, since $\operatorname{div} \theta - \frac{1}{2} \upd \Vert \upd f \Vert^2  = \langle \upd f, \nabla  \Delta f \rangle$.

Sampson's computation is based on the observation that when $(M, J)$ is a complex manifold, given a bilinear form
$B \in \Gamma(\upT^* M \otimes \upT^* M \otimes \dots)$, as an alternative to the trace $\tr B$, one can define
$\tr^c(B) \coloneqq \sum_{i=1}^M B(\bar{\varepsilon_j}, \varepsilon_j)$
% \begin{equation}
% \tr^c(B) = \sum_{i=1}^M B(\bar{\varepsilon_j}, \varepsilon_j) 
% \end{equation}
where $\varepsilon_j = {e_j}^{1,0}$. It is immediate to check that $\tr^c B = \frac{1}{4}\left[ \tr B + i \tr \left(B(J \cdot, \cdot)^{1,1}\right)\right]$
and that if $B$ is symmetric then $\tr^c(B) = \frac{1}{4} \tr B$. In particular one can define $\operatorname{div}^c(T) = \tr^c(\nabla T)$.

Reproducing the calculations leading to \eqref{eq:BochnerDivergence} by computing ${(\operatorname{div}^c)}^2(\theta)$ instead of $\operatorname{div}^2(\theta)$,
one finds:
\begin{equation} \label{eq:Firstdivc}
\operatorname{div}^c \theta = \langle \tr^c(\nabla^2 f) , \upd f(\cdot) \rangle
+ \langle \upd f(\varepsilon_l), (\nabla_{\bar{\varepsilon_l}} \upd f) (\cdot) \rangle\,.
\end{equation}
Be advised that in \eqref{eq:Firstdivc} we omit the summation symbol over the repeated index $l$; we maintain this convention
in the remainder of this section. Since $\nabla \upd f$ is symmetric (\autoref{prop:HessianSymmetric}), 
$\tr^c(\nabla^2 f) = \frac{1}{4} \tr (\nabla^2 f) = \frac{1}{4} \Delta f$. In particular, if $f$ is harmonic, \eqref{eq:Firstdivc}
reduces to 
\begin{equation} \label{eq:FirstdivcHarmonic}
\operatorname{div}^c \theta = \langle \upd f(\varepsilon_l), (\nabla_{\bar{\varepsilon_l}} \upd f) (\cdot) \rangle\,.
\end{equation}
Taking $\operatorname{div}^c$ again and using the Weitzenböck formula \eqref{eq:WeitzenbockFormula}, one finds:
 \begin{equation} \label{eq:BochnerSampson0}
 \begin{split}
 {(\operatorname{div}^c)}^2~\theta &= \langle \nabla_{\bar{\varepsilon_j}}\upd f (\varepsilon_l), \nabla_{\bar{\varepsilon_l}} \upd f(\varepsilon_j) \rangle
 + \langle \upd f(R^M(\bar{\varepsilon_j}, \bar{\varepsilon_l}) \varepsilon_j), \upd f(\varepsilon_l) \rangle\\
 & \quad -  R^N(\upd f(\bar{\varepsilon_j}), \upd f(\bar{\varepsilon_l}), \upd f(\varepsilon_j), \upd f(\varepsilon_l))\,.
  \end{split}
 \end{equation}
If $M$ is K\"ahler, then the curvature term $\langle \upd f(R^M(\bar{\varepsilon_j}, \bar{\varepsilon_l}) \varepsilon_j), \upd f(\varepsilon_l) \rangle$ vanishes. 
Indeed, this term is real (unchanged by complex conjugation by symmetry of the Riemann curvature tensor) 
and the Ricci curvature tensor of $M$ is of type $(1,1)$ (\autoref{prop:CurvatureKahler}). Sampson's Bochner formula follows:

\begin{proposition}[Sampson's Bochner formula for harmonic maps]
 Let $f \colon M \to N$ be a smooth harmonic map where $M$ is a Kähler manifold and $N$ is a Riemannian manifold. Then
 \begin{equation} \label{eq:BochnerSampson}
 {(\operatorname{div}^c)}^2~\theta = \langle \nabla_{\bar{\varepsilon_j}}\upd f (\varepsilon_l), \nabla_{\bar{\varepsilon_l}} \upd f(\varepsilon_j) \rangle
 -  R^N(\upd f(\bar{\varepsilon_j}), \upd f(\bar{\varepsilon_l}), \upd f(\varepsilon_j), \upd f(\varepsilon_l))\,.
 \end{equation}
\end{proposition}

It is remarkable that this Bochner formula does not involve the curvature of $M$.

\begin{remark}
It is easy to generalize Sampson's Bochner formula to any smooth map without the harmonicity condition by keeping track of the extra terms:
\begin{equation} \label{eq:BochnerSampsonGeneral}
\begin{split}
 {(\operatorname{div}^c)}^2~\theta &= 2 \langle \nabla_{\bar{\varepsilon_j}} (\nabla_{\bar{\varepsilon_l}} \upd f (\varepsilon_l)), \upd f(\varepsilon_j) \rangle
 + \langle \nabla_{\bar{\varepsilon_l}}\upd f (\varepsilon_l), \nabla_{\bar{\varepsilon_j}} \upd f(\varepsilon_j) \rangle\\
 & \quad + \langle \nabla_{\bar{\varepsilon_j}}\upd f (\varepsilon_l), \nabla_{\bar{\varepsilon_l}} \upd f(\varepsilon_j) \rangle
 -  R^N(\upd f(\bar{\varepsilon_j}), \upd f(\bar{\varepsilon_l}), \upd f(\varepsilon_j), \upd f(\varepsilon_l))\,.
 \end{split}
 \end{equation}
\end{remark}

Ohnita-Valli \cite{MR1069515} give a neat variant of Sampson's Bochner formula \eqref{eq:BochnerSampsonGeneral},
although they do not give any details for the proof. Define the \emph{energy form} of $f \colon M \to N$ by
$\varepsilon(f) \coloneqq \left(f^*g_N(J \cdot, \cdot)\right)^{(1,1)}$. This is a finer version of the energy density of $f$: the two are related by
$e(f) \vol_M = \varepsilon(f) \wedge \frac{\omega^{n-1}}{(n-1)!}$.

\begin{proposition}[{\cite[Eq. (1.3)]{MR1069515}}]
 Let $f \colon M \to N$ be a smooth map where $M$ is a Kähler manifold of complex dimension $n$ and $N$ is a Riemannian manifold. Then
 \begin{equation} \label{eq:BochnerOhnitaValli}
 \begin{split}
 i \partial \bar{\partial} \varepsilon(f) \wedge \frac{\omega^{n-2}}{(n-2)!} &= 
 \Big[\Vert \nabla^{0,1} \upd^{1,0} f\Vert^2 - \Vert \tr (\nabla^{0,1} \upd^{1,0} f) \Vert^2\\
 &\quad -  R^N(\upd f(\varepsilon_j), \upd f(\varepsilon_l), \upd f(\bar{\varepsilon_j}), \upd f(\bar{\varepsilon_l}))\Big] \vol_M\,.
\end{split}
 \end{equation}
\end{proposition}

\subsection{Proof of the Siu--Sampson theorem}

\subsubsection{Using Sampson's Bochner formula}

Just like Eells--Sampson's rigidity \autoref{thm:EellsSampsonRigidity}, the Siu--Sampson  \autoref{thm:SiuSampson} is simply obtained by integrating the Bochner formula 
over $M$, arguing that the left-hand side has zero integral, and that all terms on the right-hand side are pointwise $\geqslant 0$
and hence must vanish everywhere.

However Sampson \cite{MR833809} (also \cite{MR1019964} and \cite[Eq. (4.35)]{MR1391729}) concludes perhaps too quickly 
from the Bochner formula \eqref{eq:BochnerSampson}: it is not clear (to me!) that ${(\operatorname{div}^c)}^2~\theta$ is co-closed, or even
real.

That being said, \autoref{thm:SiuSampson} does follow immediately from the Ohnita--Valli version of the Bochner formula \eqref{eq:BochnerOhnitaValli}.
Indeed, observe that:
 \begin{itemize}
  \item It is always the case that $i \partial \bar{\partial} = \frac{1}{2} \upd \updc$ on $\C$-valued differential forms. In particular, $i \partial \bar{\partial} \varepsilon(f)$ is closed,
  and so is $i \partial \bar{\partial} \varepsilon(f) \wedge \frac{\omega^{n-2}}{(n-2)!}$ since $\upd \omega = 0$ (\ie{} $M$ is Kähler).
  \item Recall that $\upd_\nabla \updc f = 2i \delbar_\nabla \upd^{1,0} f$, in particular $\Vert \nabla^{0,1} \upd^{1,0} f\Vert^2 = 0$ 
  if and only if $f$ is pluriharmonic.
  \item Since $\nabla^2 f$ is symmetric, $\tr (\nabla^{0,1} \upd^{1,0} f) = \tr^c (\nabla^2 f) = \frac{1}{4} \tr (\nabla^2 f) = \frac{1}{4} \Delta f$.
  Therefore we have $\Vert \tr (\nabla^{0,1} \upd^{1,0} f) \Vert^2 = 0$ if and only if $f$ is harmonic.
 \end{itemize}

\subsubsection{Toledo's proof} \label{subsubsec:Toledo}

An elegant alternative proof was written by D. Toledo in Chapter 6 of \cite{MR1379330}. Let us briefly explain this argument which does not require a Bochner formula beforehand.

\begin{proof}[Proof of \autoref{thm:SiuSampson}]

First assume that $N = \R$. Let $M$ be a K\"ahler manifold and let $f \colon M \to \R$ be harmonic, 
\ie{} $\upd \updc f \wedge \omega^{n-1} = 0$ by \autoref{prop:DeltaHermitian}. We want to show that $f$ is pluriharmonic \ie{}
$\upd \updc f = 0$ (\autoref{prop:CharacterizationsPluriharmonic}). Consider the form $\eta \in  \Omega^{2n}(M, \R)$ where $n = \dim_{\C} M$ defined by
\begin{equation}
 \eta = \upd \updc f \wedge \upd \updc f \wedge \omega^{n-2}\,.
\end{equation}
By the linear algebra \autoref{lem:HodgeBilinear} whose proof we postpone, $\eta$ is pointwise $\leqslant 0$ with equality if and only if $\upd \updc f = 0$.
On the other hand $\upd \eta = 0$ since $\eta = \upd \left(\updc f \wedge \upd \updc f \wedge \omega^{n-2}\right)$ (recall that 
$\upd \omega = 0$ since $M$ is K\"ahler), so $\int_M \eta = 0$ by Stokes. We conclude that $\eta$ must vanish everywhere, and so must $\upd \updc f$.

For the general case, define $\eta \in  \Omega^{2n}(M, \R)$ by
\begin{equation} \label{eq:ToledoIdentity}
\begin{split}
 \eta &= \upd \left( \langle \upd_\nabla \updc f \wedge \updc f \rangle \wedge \omega^{n-2}\right)\\
 &= \langle \upd_\nabla^2 \updc f \wedge \updc f \rangle \wedge \omega^{n-2} + \langle \upd_\nabla \updc f \wedge \upd_\nabla \updc f \rangle \wedge \omega^{n-2}\,.
\end{split}
\end{equation}
By definition $\eta$ is exact, so $\int_M \eta = 0$. A straightforward extension of the linear algebra \autoref{lem:HodgeBilinear} shows that the term
$\langle \upd_\nabla \updc f \wedge \upd_\nabla \updc f \rangle \wedge \omega^{n-2}$ is pointwise $\leqslant 0$, with equality if and only if $\upd_\nabla \upd^c f = 0$.

It remains to discuss the curvature term $\langle \upd_\nabla^2 \updc f \wedge \updc f \rangle \wedge \omega^{n-2}$.
Let us skip some details and cite Toledo \cite{MR1748612}:
\emph{When rewritten using the definition of curvature $\upd_\nabla^2 = -R$, this term
turns out to be the average value of $R^N(\upd f(X), \upd f(Y), \upd f(\bar{X}),\upd f(\bar{Y}))$ over all unit
length decomposable vectors $X \wedge Y \in \otimes^2 \upT^{1,0} M$. This computation can be found in \cite{MR1379330}
 or in equivalent forms in \cite{MR584075, MR833809}.} 
%  The conclusion immediately follows from this claim.
 \end{proof}
 
\begin{remark}
I find that even in the detailed computation of \cite{MR1379330}, important arguments are overlooked: 
when Toledo writes $\upd_\nabla^2 = -R$, the curvature tensor $R$ should not just be the Riemann curvature tensor of $N$,
but also involve the curvature of $M$.
As explained in \autoref{subsec:HermitianSectionalCurvature}, the fact that $R^M$ ends up not playing a role is a little miracle
due to the fact that the Ricci tensor of a Kähler manifold is $(1,1)$.
\end{remark}

To complete the proof we need the following linear algebra lemma, which can be seen as a special case of the \emph{Hodge-Riemann bilinear relations},
a well-known result in K\"ahler geometry that leads to the Hodge index theorem (see \eg{} \cite{VoisinHodge}).
\begin{lemma} \label{lem:HodgeBilinear}
 Let $(V, g, J, \omega)$ be a Hermitian vector space of complex dimension $n$.
 If $\alpha \in \Lambda^{1,1} V^*$ is a $(1,1)$-form such that $\alpha \wedge \omega^{n-1} = 0$,
 then $\alpha \wedge \alpha \wedge \omega^{n-2} \leqslant 0$, with equality if and only if $\alpha = 0$.
\end{lemma}

For entertainment, we give a neat proof (mostly taken from \cite{MR1379330}) of this not-so-easy lemma. 
\begin{proof}
Define the Hermitian metric $h \coloneqq g - i\omega$, the space of complex endomorphisms
$\End_{\C}(V) \coloneqq \{f\in \End_{\R}(V) ~|~ fJ = Jf\}$, the space of self-adjoint endomorphisms
$\cH(V) \coloneqq \{f\in \End_{\C}(V) ~|~ h(f(x), y) = h(x, f(y))\}$, and the unitary group $\U(V) \coloneqq \{u\in \End_{\C}(V) ~|~ h(u(x), u(y)) = h(x, y)\}$.

An element $\alpha \in \Lambda^{1,1} V^*$ is a skew-symmetric bilinear form such that $\alpha(Jx, Jy) = \alpha(x,y)$. Let
\begin{equation} \label{eq:LinearAlgLemmaIso}
\begin{split}
 \cH(V) &\to \Lambda^{1,1} V^*\\
 f &\mapsto \omega_f
\end{split}
\end{equation}
where $\omega_f(x,y) \coloneqq \omega(fx, y)$. This isomorphism is moreover $\U(V)$-equivariant for the natural action of $\U(V)$ on $\cH(V)$ by conjugation
and on $\Lambda^{1,1} V^*$ by pullback: $(u \cdot \alpha)(x,y) \coloneqq \alpha(ux, uy)$. 

We claim that the decomposition of $\cH(V)$ as a sum of irreducible
$\U(V)$-modules is $\cH = \R \mathit{id} + \cH^0$, where $\cH^0$ is the space of traceless self-adjoint endomorphisms.
Indeed, consider the Cartan decomposition $\mathfrak{sl}(V) = \mathfrak{su}(h) \oplus \mathfrak{p}$,
where $\mathfrak{su}(V)$ is the Lie algebra of $\SU(V)$ and $\mathfrak{p} = \cH^0$. If there existed a nonzero proper $\U(h)$-invariant subspace 
$\mathfrak{m} \subseteq \mathfrak{p}$, then $\mathfrak{u} + \mathfrak{m}$ would be a nonzero proper ideal of the Lie algebra $\mathfrak{su}(V)$, contradicting its well-known simplicity.

Via the $\U(h)$-equivariant isomorphism \eqref{eq:LinearAlgLemmaIso}, we get a decomposition of $\Lambda^{1,1} V^*$ into 
irreducible $\U(h)$-modules as
% \begin{equation}
%  \Lambda^{1,1} V^* = \R \omega \oplus \Lambda_0^{1,1} V^*
% \end{equation}
$ \Lambda^{1,1} V^* = \R \omega \oplus \Lambda_0^{1,1} V^*$
where $\alpha \in \Lambda_0^{1,1} V^*$ if and only if $\alpha \wedge \omega^{n-1} = 0$. 

To conclude, consider the inner product
\begin{equation} \label{eq:HRInnerProduct}
\begin{split}
 \Lambda^{1,1} V^* \times \Lambda^{1,1} V^* &\to \Lambda^{2n} V^* \approx \R\\
 (\alpha, \beta ) & \mapsto \alpha \wedge \beta \wedge \omega^{n-2}\,.
 \end{split}
\end{equation}
This inner product is $\U(V)$-invariant, therefore in restriction to the irreducible subspace $\Lambda_0^{1,1} V^*$ it must be
either positive definite or negative definite or identically zero. Compute any example to conclude.
\end{proof}

\begin{remark}
 Toledo's proof is very closely related to the Ohnita-Valli version of Sampson's Bochner formula \eqref{eq:BochnerOhnitaValli}. Indeed, it is easy to check that
 $\varepsilon(f) = -\frac{1}{2} \langle \upd f \wedge \updc f \rangle$ and derive that $\eta = -4i \partial \bar{\partial} \varepsilon(f) \wedge \omega^{n-2}$.
%  \begin{equation}
%   \eta = -4i \partial \bar{\partial} \varepsilon(f) \wedge \omega^{n-2}\,.
%  \end{equation}
One can therefore equate Toledo's identity \eqref{eq:ToledoIdentity} with the Ohnita--Valli formula \eqref{eq:BochnerOhnitaValli}
as long as one can equate the curvature terms, which is done in \cite{MR1379330}, and upgrade
\autoref{lem:HodgeBilinear}: the inner product \eqref{eq:HRInnerProduct} is actually given explicitly by
\begin{equation}
 \alpha \wedge \alpha \wedge \frac{\omega^{n-2}}{(n-2)!} = (\tr \alpha)^2 - \Vert \alpha \Vert^2\,.
\end{equation}
This identity can be proved by brute force: take
an orthonormal basis $(e_1, J e_1, \dots)$ of $V$, etc.
\end{remark}

\subsection{First applications to rigidity}
\label{subsec:FirstRigidity}

This first straightforward consequence of \autoref{thm:SiuSampson} is due to Sampson \cite{MR833809}:
\begin{theorem} \label{thm:SampsonHn}
Let $f \colon M \to N$ where $M$ is compact Kähler and $N$ has constant negative sectional curvature. If $f$
is harmonic, then the rank of $f$ is everywhere $\leqslant 2$.
\end{theorem}

\begin{proof}
If $N$ has constant curvature $k$, then $R^N(X, Y, \bar{X}, \bar{Y}) = k \Vert X \wedge Y \Vert^2$ for all $X, Y \in \upT_{\C} N$.
By \autoref{thm:SiuSampson}, $\upd f(X^{1,0})$ and $\upd f(Y^{1,0})$ must be collinear for any $X, Y \in \upT M$. This implies 
$\operatorname{rank}(f) \leqslant 2$.
\end{proof}

\begin{remark}
 Combining with the existence theorem of Eells--Sampson (\autoref{thm:EellsSampson}), we obtain topological restrictions on maps $f \colon M \to N$ from a compact Kähler manifold to a compact hyperbolic manifold, such as: $f_* \colon H_k(M, \Z) \to H_k(N, \Z)$ is zero if $k>2$.
\end{remark}

We shall recover \autoref{thm:SampsonHn} with a more algebraic proof in \autoref{subsec:SymmSiu}. Next, 
we prove the \emph{strong rigidity theorem of Siu} for strongly negatively curved Kähler manifolds:
\begin{theorem} \label{thm:SiuRigidity1}
Let $M$ and $N$ be compact Kähler manifolds, with $N$ strongly negatively curved. Any harmonic map $M \to N$ is $\pm$-holomorphic unless 
it has rank $\leqslant 2$ everywhere.
\end{theorem}
\begin{corollary} \label{cor:SiuRigidity2}
Let $N$ be a compact Kähler manifold of complex dimension $\geqslant 2$ with strongly negative curvature. 
Any compact Kähler manifold homotopy equivalent to $N$ must be either $\pm$-biholomorphic to $N$.
\end{corollary}

\begin{proof}[Proof of \autoref{thm:SiuRigidity1} and \autoref{cor:SiuRigidity2}]
 If $f$ is not $\pm$-holomorphic at $x$, there exist $X_1, X_2 \in \upT_x^{1,0} M$ such
 that $\upd f(X_1)$ is not of type $(1,0)$ and $\upd f(X_2)$ is not of type $(0,1)$. One can then find $X_0 \in \upT_x^{1,0} M$
 such that $\upd f(X_0)$ is of mixed type (either $X_0 = X_1 + X_2$ or $X_0 = X_2$ works). Given any $Y \in \upT_x^{1,0} M$,
 consider $\sigma \coloneqq \upd f (X_0)^{1,0} \otimes \upd f (Y)^{0,1} -\upd f (Y)^{1,0} \otimes\upd f (X_0)^{0,1}$. If $f$ is harmonic,
 then by \autoref{thm:SiuSampson} $R(\upd f(X_0), \upd f(Y), \overline{\upd f(X_0)}, \overline{\upd f(Y)}) = Q(\sigma, \bar{\sigma}) = 0$.
 By strongly negative curvature of $N$, it follows $\sigma = 0$. This implies that $\upd f(Y)$ and $\upd f(X_0)$ are collinear over $\C$,
 and since this is true for all $Y$, the rank of $f$ at $x$ is $\leqslant 2$. If $f$ is not holomorphic (resp.\ antiholomorphic) on $M$, 
 the set $U$ (resp.\ $V$) where $\upd f( \cdot )^{0,1} \neq 0$ (resp.\ $\upd f( \cdot )^{1,0} \neq 0$)
 is open dense. Thus $f$ has rank $\leqslant 2$ on the open dense set $U \cap V$, and hence everywhere.
 
 For \autoref{cor:SiuRigidity2}, let $f \colon M \to N$ be a homotopy equivalence. By \autoref{thm:EellsSampson} we may assume
 $f$ is harmonic. Since $M$ and $N$ are closed, they must have same dimension and $f$ has degree $1$. In particular $f$ is surjective 
 and has full rank $n \geqslant 4$ somewhere. By \autoref{thm:SiuRigidity1}, $f$ is $\pm$-holomorphic. If $f$ is not injective,
 then it must send a subvariety $Z \subseteq M$ of positive dimension to a point. This is not possible because the fundamental class of $Z$
 is a nontrivial cohomology class of $M$ (the integral of a power of the Kähler form on $Z$ is the volume of $Z$, hence positive) and $f$
 induces an isomorphism on cohomology. 
\end{proof}

\begin{remark}
 \autoref{cor:SiuRigidity2} says that $N$ is \emph{strongly rigid} as a Kähler manifold. This echoes the celebrated rigidity theorem of Mostow; this will be discussed
 in \autoref{subsec:Mostow}.
\end{remark}

\begin{remark}
We refer to  \cite{MR1363515} and \cite{MR1082880} for further rigidity results concerning complex manifolds via the harmonic maps approach.
\end{remark}

\section{Applications to symmetric spaces} \label{sec:SymmetricSpaces}

We now apply the theorem of Siu and Sampson to the situation where the target Riemannian manifold is locally symmetric. We derive in particular a second version of Siu's strong rigidity theorem,
and explain its relation to Mostow rigidity. 

\begin{remark}
This section, especially \autoref{subsec:AbelianSubalgebras} and \autoref{subsec:SymmSiu}, is largely based on Chapter 6 (mostly written by D. Toledo) of the excellent book \cite{MR1379330} of 
Amorós, Burger, Corlette, Kotschick, and Toledo.
\end{remark}

\subsection{Symmetric spaces of noncompact type}
\label{subsec:SymmSpaces}

Let $\mathbf{X} = G/K$ be a symmetric space of noncompact type: $G$ is a semisimple Lie group without compact factors and $K$ is a maximal compact
subgroup\footnote{$G$ is assumed connected by definition of semisimple, and without loss of generality we can assume it has trivial center.}. Denote by $\mathfrak{g}$ and $\mathfrak{k}$ the Lie algebras of $G$ and $K$ respectively, and denote $B$ the Killing form on $\mathfrak{g}$.
Recall that $B$ is an $\operatorname{Ad} G$-invariant symmetric bilinear form on $\mathfrak{g}$, and it is nondegenerate because $\mathfrak{g}$ is semisimple.
We have the Cartan decomposition
\begin{equation}
 \mathfrak{g} = \mathfrak{k} \oplus \mathfrak{p}
\end{equation}
where $\mathfrak{p}$ is the $B$-orthogonal of $\mathfrak{k}$. This decomposition satisfies $[\mathfrak{k}, \mathfrak{k}] \subseteq \mathfrak{k}$, $[\mathfrak{k}, \mathfrak{p}] \subseteq \mathfrak{p}$, 
and $[\mathfrak{p}, \mathfrak{p}] \subseteq \mathfrak{k}$. Moreover $B$ is positive definite on $\mathfrak{p}$ and
negative definite on $\mathfrak{k}$. Thus $B$ induces an inner product on $\upT_p \mathbf{X} \approx \mathfrak{p}$ for any $p \in \mathbf{X}$; this globally defines a $G$-invariant Riemannian metric on $\mathbf{X}$.

The curvature tensors (see \autoref{subsec:CurvatureTensors}) on $\mathbf{X}$ have nice expressions:
\begin{proposition} \label{prop:CurvSymmSpace}
Let $o = eK \in \mathbf{X} = G/K$. For all $X, Y, Z, W \in \upT_{o} \mathbf{X} \approx \mathfrak{p}$, we have:
\begin{itemize}
 \item Riemann curvature tensor: $R(X,Y) Z= -[[X,Y],Z]$
 \item purely covariant  Riemann tensor: $R(X,Y, Z, W)  = -B[[X,Y], [Z,W]]$
 \item curvature operator: $Q (X \otimes Y, X \otimes Y) = -B[[X,Y] , [X,Y]]$
\end{itemize}
\end{proposition}

\begin{proof}
Since $G$ is the group of isometries of $\mathbf{X}$, $\mathfrak{g}$ can be identified to the space of Killing vector fields. We indicate by $\bar{X}$ the 
Killing field associated to $X \in \mathfrak{g}$. One easily checks that $[\bar{X},\bar{Y}] = -\overline{[X,Y]}$, where the Lie bracket is that of vector fields on the left-hand side.
Let $o = eK \in \mathbf{X}$. A quick computation shows that the Levi-Civita connection of $\mathbf{X}$ at $o$ is given by $\left(\nabla_X T\right)_o = [\bar{X}, T]_o$.
Note that Killing fields vanishing at $o$ correspond to the Lie subalgebra $\mathfrak{k} \subseteq \mathfrak{g}$, while Killing fields $\bar{X}$ such that $\left(\nabla \bar{X} \right)_o = 0$
correspond to $X \in \mathfrak{p} \subseteq \mathfrak{g}$ (called \emph{infinitesimal transvections}).

Since $R$ is tensorial, we have, at $o$:
\begin{equation}
 R(X,Y)Z = R(\bar{X}, \bar{Y}) \bar{Z} = \nabla_{\bar{X}} \nabla_{\bar{Y}} \bar{Z} - \nabla_{\bar{Y}} \nabla_{\bar{X}} \bar{Z} - \nabla_{[\bar{X}, \bar{Y}]} \bar{Z}\,.
\end{equation}
The last term is zero because $[\bar{X}, \bar{Y}] = - \overline{[X,Y]} = 0$ at $o$ since $[X,Y] \in \mathfrak{k}$.
The first term is rewritten  $\nabla_{\bar{X}} \nabla_{\bar{Y}} \bar{Z} = [\bar{X}, \nabla_{\bar{Y}} \bar{Z}] = \nabla_{[\bar{X}, \bar{Y}]} \bar{Z} + \nabla_{\bar{Y}}[\bar{X}, \bar{Z}]$
---by some sort of Leibniz rule for the Lie derivative along a Killing field. Again the first term vanishes, so 
$\nabla_{\bar{X}} \nabla_{\bar{Y}} \bar{Z} = \nabla_{\bar{Y}}[\bar{X}, \bar{Z}] = [\bar{Y},[\bar{X}, \bar{Z}]]$.
We thus find
\begin{equation}
 \begin{split}
  R(X,Y)Z &= [\bar{Y},[\bar{X}, \bar{Z}]] - [\bar{X},[\bar{Y}, \bar{Z}]]\\
  &= [\bar{Z},[\bar{X}, \bar{Y}]]
 \end{split}
\end{equation}
by the Jacobi identity. Conclude that $R(X,Y) Z= [Z,[X,Y]]$ using the property $[\bar{U},\bar{V}] = -\overline{[U,V]}$. The expression of the purely covariant Riemann tensor
and the curvature operator quickly follows.
\end{proof}

\begin{remark}
The proof above is taken from \cite{PecastaingSymmetricSpaces}. The reader will find alternative proofs in classical textbooks such as
\cite{MR1393941, MR1834454,  PaulinNotes}.
\end{remark}

\begin{remark}
The fact that $\mathbf{X} = G/K$ is a Riemannian symmetric space, \ie{} there exists a central symmetry about every point, easily implies that $\nabla R = 0$. 
Conversely, any Riemannian manifold whose curvature tensor is parallel is locally isometric to a symmetric space $G/K$ (not necessarily of noncompact type).
\end{remark}

A seemingly innocuous consequence of \autoref{prop:CurvSymmSpace} is that the complexified curvature operator is also given by 
$Q (X \otimes Y, X \otimes Y) = -B^\C[[X,Y] , [X,Y]]$, where $B^\C$ denotes the complex linear extension of the Killing form to $\mathfrak{p}^\C \coloneqq \mathfrak{p} \otimes \C$.
This not only implies that $Q(\sigma, \bar{\sigma}) \leqslant 0$ for all $\sigma \in \medotimes^2 \mathfrak{p}^\C$; furthermore we have $R(X,Y) = 0$ if $Q (X \otimes Y, X \otimes Y) = 0$ for any $X, Y \in \mathfrak{p}^\C$, which is exceptionally strong. Since $R$ is parallel, this must hold at any point of $\mathbf{X}$, therefore we obtain:
\begin{corollary} \label{cor:CurvSymmSpace}
Let $\mathbf{X} = G/K$ be a symmetric space of noncompact type.
\begin{enumerate}[(i)]
 \item \label{cor:CurvSymmSpacei} $\mathbf{X}$ has very strongly nonpositive curvature. 
 \item \label{cor:CurvSymmSpaceii} At any $p \in \mathbf{X}$, we have, for all $X, Y \in \upT_p \mathbf{X} \otimes \C$: if $R(X, Y, X, Y) = 0 $ then $R(X,Y) = 0$.
\end{enumerate}
\end{corollary}

\begin{remark}
A \emph{Cartan subspace} of $\mathfrak{g} = \mathfrak{k} + \mathfrak{p}$ is defined to be a maximal abelian subspace of $\mathfrak{p}$. All Cartan subspaces have the same dimension, called the \emph{rank} of $\mathbf{X}$. It follows from \autoref{cor:CurvSymmSpace} that the rank of $\mathbf{X}$ is the maximal dimension of a flat subspace of its tangent space.
\end{remark}

\begin{example}
 Hyperbolic space $\mathbf{X} = \H^n$, where $G = \SO^+(n,1)$ and $K = \SO(n) \times \{1\}$, is a symmetric space of noncompact type of rank $1$. As is well-known, $\bH^n$ is isotropic with constant sectional curvature $-1$. This confirms \ref{cor:CurvSymmSpacei}, and note that \ref{cor:CurvSymmSpaceii} never arises unless $X$ and $Y$ are collinear.
\end{example}

Hermitian symmetric spaces are another important class of examples, especially for the applications of the Siu--Sampson theorem. $\mathbf{X} = G/K$ is called \emph{Hermitian symmetric} is there exists a $G$-invariant, equivalently parallel, compatible complex structure $J$ (in particular $\mathbf{X}$ Kähler). One can show that $\mathbf{X} = G/K$ is Hermitian symmetric if and only if $K$ has nondiscrete center \cite[Chap.\ VIII, \S6]{MR1834454}.

\begin{remark}
If $\mathbf{X} = G/K$ is irreducible, \ie{} $G$ is simple, the complex structure $J$ is unique up to sign, and the center of $K$ is a copy of $\U(1)$. Otherwise, $J$ is determined up to sign on each irreducible factor.
\end{remark}

\begin{example}
A bounded domain $\Omega \subseteq \C^n$ is called a \emph{bounded symmetric domain} if for all $x \in \Omega$, there exists an involutive biholomorphism of $\Omega$ with $x$ as an isolated fixed point. On such a domain, one can define a Hermitian metric called the \emph{Bergman metric}
for which biholomorphisms of $\Omega$ are isometries. Letting $G$ denote the biholomorphism group of $\Omega$ and $K$ the stabilizer of a point,
$\mathbf{X} = G/K \approx \Omega$ is a Hermitian symmetric space of noncompact type, whose Hermitian structure is given by the Bergman metric on $\Omega$.
Conversely, the \emph{Harish-Chandra embedding} exhibits every Hermitian symmetric space of noncompact type as a bounded symmetric domain \cite[Thm~7.1]{MR1834454}. As a fundamental example, taking the unit ball in $\C^{n}$ for $\Omega$ yields the ball model of complex hyperbolic space $\H_\C^n$.
\end{example}

\subsection{\texorpdfstring{Abelian subalgebras of $\mathfrak{p}^\C$}{Abelian subalgebras of p\textasciicircum{C}}}
\label{subsec:AbelianSubalgebras}

Let $\mathbf{X} = G/K$ be a symmetric space of noncompact type as in \autoref{subsec:SymmSpaces}. Recall that we have the Cartan decomposition 
$\mathfrak{g} = \mathfrak{k} \oplus \mathfrak{p}$. In preparation of \autoref{subsec:SymmSiu}, let us investigate abelian complex subspaces $\mathfrak{a} \subseteq \mathfrak{p}^\C$, especially their dimension. 
Of course, we always have $\dim_\C \mathfrak{a} \leqslant \dim_\C \mathfrak{p}^\C = \dim \mathbf{X}$. Can a better upper bound be achieved? Taking $\mathfrak{a} = \mathfrak{a}_0^\C$ where $\mathfrak{a}_0 \subseteq \mathfrak{p}$ is a Cartan subspace shows that one cannot hope for better than $\dim_\C \mathfrak{a} \leqslant \operatorname{rank}(\mathbf{X})$. First let us look at a simple example where this bound does hold:
\begin{lemma} \label{lem:AbelianHn}
 Let $\mathbf{X} = \H^n$. If $\mathfrak{a} \subseteq \mathfrak{p}^\C$ is abelian, then $\dim_\C \mathfrak{a} \leqslant 1$.
\end{lemma}

\begin{proof}
As an instructive exercise, let us write both an algebraic and a geometric proof. For the algebraic proof, let $\mathfrak{g} = \mathfrak{so}(n,1)$. Generally, elements of $\mathfrak{so}(p,q)$ are matrices  $M = \begin{bmatrix}
 A & B\\
 \transpose{B} & D
\end{bmatrix} \in \mathfrak{sl}(p+q, \R)$ with $\transpose{A} = -A$ and $\transpose{D} = -D$. The Cartan decomposition $\mathfrak{g} = \mathfrak{k} \oplus \mathfrak{p}$ consists in writing $M = 
\begin{bmatrix}
 A & 0\\
 0 & D
\end{bmatrix} + \begin{bmatrix}
 0 & B\\
 \transpose{B} & 0
\end{bmatrix}$.
Thus $\mathfrak{p}^\C$ consits of matrices of the form 
$\begin{bmatrix}
 0 & B\\
 \transpose{B} & 0
\end{bmatrix}$ where $B$ is a $p \times q$ matrix with complex coefficients. When $q=1$, two such matrices
commute if and only if they are collinear, and the result follows.

Now  the geometric proof. Since $\H^n$ has sectional curvature $-1$,
its Riemann curvature tensor is given by $\langle R(X,Y)Z, W \rangle = \langle X, Z \rangle \langle Y, W\rangle - \langle Y, Z \rangle \langle X, W\rangle$. A neat way to interpret this identity is that $R(X,Y)$ is the skew-adjoint endomorphism of $\upT \mathbf{X}$
associated to $X^\flat \wedge Y^\flat \in \Lambda^2 \upT^* \mathbf{X}$, where $X^\flat \in \upT^* \mathbf{X}$ indicates the metric dual of $X \in \upT \mathbf{X}$. The identification $R(X,Y) = X^\flat \wedge Y^\flat$ extends to $\upT_{\C} \mathbf{X} $. It follows that, for all $X, Y \in \upT_{\C} \mathbf{X} $, $R(X,Y) = 0$ if and only if $X$ and $Y$ are collinear. The conclusion ensues.
\end{proof}

In the example above, all elements of $\mathfrak{p}^\C$ are semisimple. One can show that if $\mathfrak{a}$ only has semisimple elements, it is conjugate to $\mathfrak{a}_0^\C$ for some abelian subspace $\mathfrak{a}_0 \subseteq \mathfrak{p}$, therefore $\dim_\C \mathfrak{a} \leqslant \operatorname{rank}(\mathbf{X})$. However, in general, $\mathfrak{a}$ can contain nilpotent elements, and its dimension can be much higher.

Indeed, assume $\mathbf{X}$ is Hermitian symmetric. The linear complex structure $J$ in $\upT_o \mathbf{X} \approx \mathfrak{p}$ yields a decomposition of $\mathfrak{p}^\C$ into $\pm i$-eigenspaces as $\mathfrak{p}^\C = \mathfrak{p}^{1,0} \oplus \mathfrak{p}^{0,1}$. Since $J$ is integrable, we must have\footnote{On an almost complex manifold $(\mathbf{X},J)$, the involutivity of the distribution $\upT^{1,0}\mathbf{X}$, \ie{} its stability under the Lie bracket, is equivalent to the vanishing of the Nijenhuis tensor.} $[\mathfrak{p}^{1,0}, \mathfrak{p}^{1,0}] \subseteq \mathfrak{p}^{1,0}$. On the other hand, $[\mathfrak{p}, \mathfrak{p}] \subseteq \mathfrak{k}$ implies that $[\mathfrak{p^\C}, \mathfrak{p^\C}] \subseteq \mathfrak{k^\C}$; therefore $[\mathfrak{p}^{1,0}, \mathfrak{p}^{1,0}] = 0$. Thus, for Hermitian symmetric spaces, $\mathfrak{a} = \mathfrak{p}^{1,0}$ shows that $\dim_\C \mathfrak{a} = \frac{1}{2}\dim \mathbf{X}$ is possible.

\begin{example}
It is instructive to work out the example $\mathbf{X} = \SU(p,q) / \operatorname{S}(\U(p) \times \U(q))$, similarly to $\mathfrak{so}(p,q)$ above. This is done in \cite[Example 6.9]{MR1379330}.
\end{example}

The following result of Carlson--Toledo shows that the bound $\dim_\C \mathfrak{a} \leqslant \frac{1}{2}\dim \mathbf{X}$ holds for all symmetric spaces
of noncompact type, and is only attained for Hermitian symmetric spaces:
\begin{theorem}[Carlson--Toledo \cite{MR1019964}] \label{thm:CarlsonToledo}
Let $\mathbf{X} = G/K$ be a symmetric space of noncompact type. If $\mathfrak{a} \subseteq \mathfrak{p}^\C$ is an abelian subalgebra, then
$\dim_\C \mathfrak{a} \leqslant \frac{1}{2} \dim \mathbf{X}$. Moreover, if equality holds, $\mathbf{X}$ is Hermitian symmetric,
and provided $\mathbf{X}$ has no $\H^2$ factor\footnote{$\mathbf{X}$ has no $\H^2$ factor if it is not isometric to $\H^2 \times \mathbf{X}_1$ where $\mathbf{X}_1$ is a symmetric space. Equivalently, $G$ has no $\PSL(2,\R)$ factor.} we have $\mathfrak{a} = \mathfrak{p}^{1,0}$ for some invariant complex structure on $\mathbf{X}$.
\end{theorem}

\begin{proof}
 We only sketch the proof and refer to \cite{MR1019964} for details. By \autoref{lem:AbelianHn}, the results holds for $\H^2$. If $\mathbf{X}$ has no $\H^2$ factor, a technical lemma ensures that if $\mathfrak{a}$ contains a nonzero semisimple element, then $\dim_\C \mathfrak{a} < \frac{1}{2} \dim \mathbf{X}$. 
%  Assume $\mathfrak{a}$ is maximal. The semisimple part of any element $X \in \mathfrak{a}$ commutes with any element that commutes with $X$, therefore it also belongs to $\mathfrak{a}$. Thus, by  the previous lemma, we may assume that $\mathfrak{a}$ only contains nilpotent elements.
Thus we may assume that $\mathfrak{a}$ has no semisimple elements. Since real elements are semisimple, we have $\mathfrak{a} \cap \bar{\mathfrak{a}} = 0$, and if $\dim_\C \mathfrak{a} = \frac{1}{2} \dim_\C \mathfrak{p}^\C$ then $\mathfrak{p}^\C = \mathfrak{a} \oplus \bar{\mathfrak{a}}$. This means
that $\mathfrak{a} = \mathfrak{p}^{1,0}$ for some linear complex structure on $\mathfrak{p}$. One readily shows that it is compatible with the Killing form and invariant, \ie{} $[\mathfrak{k}, \mathfrak{a}] \subseteq \mathfrak{a}$. 
\end{proof}

\subsection{Application of the Siu--Sampson theorem}
\label{subsec:SymmSiu}

Let us go back to the setting of the theorem of Siu and Sampson: let $f \colon M \to N$ be a harmonic map where $M$ is compact Kähler and $N$ is Riemannian. 
We now specialize to the case where $N$ is locally symmetric of noncompact type, \ie{} locally isometric to a symmetric space of noncompact type $\mathbf{X} = G/K$.

\begin{remark}
If $N$ is additionally assumed complete, its universal cover $\tilde{N}$ must be isometric to $\mathbf{X} = G/K$. 
In other words, $N = \Gamma \setminus \mathbf{X}$ where  $\Gamma < G$ is a discrete subgroup acting freely and properly on $\mathbf{X}$---it is enough that $\Gamma < G$ is discrete and torsion-free\footnote{It is a general fact that any isometric action of a discrete group
in a proper metric space is proper \cite[Thm.\ 5.3.5]{MR2249478}.}.
When $N$ is not assumed complete, we merely have a \emph{developing map} $\tilde{N} \to \mathbf{X}$ which is a local isometry, 
but neither surjective nor a covering map in general\footnote{This is the language of $(X,G)$-structures, refer to \eg{} \cite{MR2827816} for more developments.}.
\end{remark}

Since $\mathbf{X}$, hence $N$, has very strongly nonpositive curvature (see \autoref{cor:CurvSymmSpace}), we may apply \autoref{thm:SiuSampson} and conclude that $f$ is pluriharmonic, and moreover $R^N\left(X, Y, \bar{X}, \bar{Y}\right) = 0$ for all $x\in M$ and $X, Y \in \upd f(\upT_x^{1,0} M)$. By \autoref{cor:CurvSymmSpace} \ref{cor:CurvSymmSpaceii}, it follows that $R^N(X, Y) = 0$. 

For $x \in M$, denote $\mathfrak{a}_x \coloneqq \upd f\left(\upT_x^{1,0} M\right) \subseteq \upT_{f(x)} N \otimes \C$.
Note that $\operatorname{rank}_x(f) = \dim_\R \upd f\left(\upT_x M\right) = \dim_\C (\mathfrak{a}_x + \bar{\mathfrak{a}}_x)$, therefore $\operatorname{rank}_x(f) \leqslant 2 \dim_\C \mathfrak{a}_x$.
Identifying $\upT_{f(x)} N \approx \mathfrak{p}$, we have $\mathfrak{a}_x \subseteq \mathfrak{p}^\C$, 
and the fact  that $R^N(X, Y) = 0$ for all $X, Y \in \upd f(\upT_x^{1,0} M)$ translates to $[\mathfrak{a}_x, \mathfrak{a}_x] = 0$, \ie{} $\mathfrak{a}_x$ is an abelian subalgebra of $\mathfrak{p}^\C$. Let us collect these findings:
\begin{lemma} \label{lem:AbelianSub}
Let $f \colon M \to \mathbf{X}$ be a harmonic map, where $M$ is compact Kähler and $\mathbf{X} = G/K$ is (locally) symmetric  of noncompact type. If $f$ is harmonic, then $f$ is pluriharmonic. Moreover, for all $x \in M$, $\mathfrak{a}_x \coloneqq \upd f(\upT_x^{1,0} M)$ is an abelian subalgebra of $\mathfrak{p}^\C$.
\end{lemma}

We can now reinvest the analysis of the dimension of such subalgebras performed in \autoref{subsec:AbelianSubalgebras}.
Let us start with the case where $N$ is locally isometric to $\mathbf{X} = \H^n$. Equivalently, $N$ has constant sectional curvature $-1$.
We recover \autoref{thm:SampsonHn} as an immediate consequence of \autoref{lem:AbelianHn}:
\begin{theorem}[Sampson \cite{MR833809}] \label{thm:SampsonHn2}
Let $f \colon M \to N$ be a harmonic map where $M$ is compact Kähler and $N$ is hyperbolic. Then $\operatorname{rank}(f) \leqslant 2$ everywhere. 
\end{theorem}

\begin{remark}
 \autoref{thm:SampsonHn2} was refined by Carlson--Toledo \cite{MR1019964}: they proved that either $f$ maps to a closed geodesic, or factorizes as $f = \psi \circ \varphi$,
 where $\varphi \colon M \to C$ is a holomorphic map to a Riemann surface $C$ and $\psi \colon C \to N$ is harmonic. 
\end{remark}

Next, as a consequence of \autoref{thm:CarlsonToledo} we obtain:
\begin{theorem} \label{thm:SiuHarmonic}
Let $f \colon M \to N$ be a harmonic map where $M$ is compact Kähler and $N$ is locally symmetric of noncompact type modelled on $\mathbf{X} = G/K$.
If $\operatorname{rank}_x(f) = \dim N$ for some $x\in N$, then $\mathbf{X}$ is Hermitian symmetric. 
Moreover, if $\mathbf{X}$ has no $\H^2$ factor, $f$ is holomorphic for some invariant complex structure on $\mathbf{X}$.
\end{theorem}

\begin{proof}
Let $\mathfrak{a}_x = \upd f\left(\upT_x^{1,0} M\right) = \Imag(\upd_x^{1,0} f)$. By the previous discussion, if $f$ is harmonic then it is pluriharmonic and $\mathfrak{a}_x$ is abelian. Moreover, since $\operatorname{rank}_x(f) \leqslant 2 \dim_\C \mathfrak{a}_x$, if
$\operatorname{rank}_x(f) = \dim N$ then $\dim_\C \mathfrak{a}_x \geqslant \frac{1}{2} \dim N$. In this case \autoref{thm:CarlsonToledo} says that $\mathbf{X}$ is Hermitian symmetric, and if it has no $\H^2$ factor then $\mathfrak{a}_x = \mathfrak{p}^{1,0}$ for some invariant complex structure on $\mathbf{X}$. In other words, $f$ is holomorphic at $x$. 

Let $U = \{x \in M ~\colon \dim_{\C} \mathfrak{a}_x = \frac{1}{2} \dim N\}$.
By assumption, $U$ is nonempty. 
By \autoref{prop:CharacterizationsPluriharmonic}, $\upd^{1,0} f$ is a holomorphic section of $f^* \upT N \otimes \C$, which is a holomorphic vector bundle as we shall see in \autoref{subsec:NAH}. 
Therefore the complement of $U$, which consists of the points $x \in M$ where $\upd^{1,0} f$ does not have maximal rank, is the vanishing locus of a holomorphic function, 
so it is a proper subvariety. In particular, $U$ is dense. We know that $f$ is holomorphic on $U$, \ie{} $\delbar f$ vanishes on $U$, hence everywhere.
\end{proof}

\begin{remark}
\autoref{thm:SiuHarmonic} can be refined by trading the assumption $\operatorname{rank}_x(f) = \dim N$ for $\operatorname{rank}_x(f) > 2 p(\mathbf{X})$, where
$p(\mathbf{X})$ is the maximal dimension of a subalgebra of $\mathfrak{p}^\C$ not contained in $\mathfrak{p}^{1,0}$.
Siu \cite[Thm.\ 6.7]{MR658472} calculated $p(\mathbf{X})$ when $\mathbf{X}$ is Hermitian, Carlson--Toledo \cite{MR1214704} for $\mathbf{X}$ non-Hermitian symmetric of classical type, 
and Carlson--Hern\'{a}ndez \cite{MR1129347} for the Cayley hyperbolic plane.
\end{remark}

As in \autoref{subsec:FirstRigidity}, we can combine this result with the theorem of Eells--Sampson (\autoref{thm:EellsSampson}) to derive a second version of the strong rigidity theorem of Siu (compare with \autoref{thm:SiuRigidity1}):
\begin{theorem} \label{thm:SiuRigidity2}
 Let $f \colon M \to N$ be a homotopy equivalence where $M$ is compact Kähler and $N = \Gamma \setminus \mathbf{X}$ compact locally symmetric, where 
 $\mathbf{X} = G/K$ is a Hermitian symmetric space of noncompact type with no $\H^2$ factor.
 Then $f$ is homotopic to a biholomorphism for some invariant complex structure on $\mathbf{X}$.
\end{theorem}

\begin{proof}
By \autoref{thm:EellsSampson}, we may assume $f$ harmonic. By \autoref{thm:SiuHarmonic}, $\mathbf{X}$ is Hermitian symmetric and 
$f$ is holomorphic if $X$ has no $\H^2$ factor. It remains to prove that $f$ is bijective: this is the same topological argument as 
in the proof of \autoref{cor:SiuRigidity2}.
\end{proof}

\begin{remark}
 The assumption that $\mathbf{X}$ has no $\H^2$ factor is essential in \autoref{thm:SiuRigidity2}. Indeed, consider the case $\mathbf{X} = \H^2$.
 Let $X$ and $Y$ be any two closed Riemann surfaces having same genus $g > 1$. The Poincaré metric\footnote{This is the unique conformal metric of constant curvature $-1$, 
 given by the celebrated \emph{uniformization theorem}.} on $Y$ gives it a hyperbolic structure. Since $X$ and $Y$ have the same topology, there exists a homeomorphism $f \colon X \to Y$.
 If \autoref{thm:SiuRigidity2} applied, $f$ would be homotopic to a biholomorphism. However, in general, $X$ and $Y$ are not biholomorphic: this is the starting point
 of \emph{Teichmüller theory}, which says that the moduli space of complex structures is $(3g-3)$-complex dimensional.
\end{remark}

\begin{remark}
When $\mathbf{X}$ is of compact type, pluriharmonic maps $f \colon M \to N$ are still $\pm$-holomorphic,
however harmonic maps are not always pluriharmonic. Nevertheless, interesting results can be obtained under more or less general restrictions: refer to 
\cite{MR960100, MR1082880} for more details.
\end{remark}

\subsection{Relation to Mostow rigidity}
\label{subsec:Mostow}

We now explain the relation between Siu's strong rigidity and the famous \emph{Mostow rigidity}. In its original form 
\cite{MR236383}, the theorem of Mostow says that any closed hyperbolic manifold $M = \Gamma \setminus \H^n$ of dimension $n>2$ is \emph{strongly rigid}, in the sense that 
$M$ is isometric to any other closed hyperbolic manifold homotopic to it. Algebraically, this theorem translates to the fact any two
uniform lattices\footnote{A \emph{lattice} $\Gamma < G$ in a Lie group is a discrete subgroup such that $\Gamma \setminus G$ has finite volume. It is called  \emph{uniform}
if  $\Gamma \setminus G$ is compact. Prasad \cite{MR385005} extended Mostow rigidity to lattices that are not assumed uniform.} $\Gamma_1, \Gamma_2 < G = \O^+(n,1)$ that are isomorphic as groups are conjugate in $G$.
Mostow then generalized his theorem in \cite{MR0385004} to any closed locally symmetric manifold of noncompact type with no $2$-dimensional hyperbolic factor.
This is the geometric form of Mostow's theorem; let us record the equivalent algebraic form:
\begin{theorem}[Mostow rigidity] \label{thm:MostowAlgebraic}
Let $G$ be a semisimple Lie group with trivial center without compact factors nor any $\PSL(2,\R)$ factor. If $\Gamma_1, \Gamma_2 < G$ are uniform lattices that are isomorphic as groups, then they are conjugate subgroups of $G$.
\end{theorem}

Siu's strong rigidity theorem, version \autoref{thm:SiuRigidity2}, may be seen as a generalization of Mostow rigidity when $G$ is Hermitian.
Indeed, assume both $M = \Gamma_1 \setminus \mathbf{X}$ and $N = \Gamma_2 \setminus \mathbf{X}$ are closed manifolds locally isometric to $\mathbf{X} = G/K$. For simplicity, assume $G$ is simple, \ie{} $\mathbf{X}$ is irreducible. If $f \colon M \to N$ is a homotopy equivalence, then by \autoref{thm:SiuRigidity2} $f$ is homotopic to a $\pm$-biholomorphism. Lift $f$ to a $(\Gamma_1, \Gamma_2)$-equivariant biholomorphism $\tilde{f} \colon \mathbf{X} \to \mathbf{X}$. Since $\pm$-biholomorphisms of $\mathbf{X}$ are isometries, as shows the Bergman metric on a bounded symmetric domain, $\tilde{f}$ coincides with an element of $G$.

\medskip

Let us conclude with a few historical comments on the rigidity of lattices. Fore more details, we refer to \cite{MR1168043, MR2090772, MR1321644}.
Before Mostow's results, the \emph{local} rigidity\footnote{The \emph{(local) rigidity} of $\Gamma < G$ means that if $\rho_t \colon \Gamma \to  G$ is a one-parameter family of injective group homomorphisms, 
where $\rho_0$ is the inclusion, such that $\rho_t(\Gamma) < G$ is a uniform lattice for all $t$, then $\rho_t$ is conjugate to $\rho_0$ in $G$.} of lattices in a group $G$ as in \autoref{thm:MostowAlgebraic}
had been proven by Weil \cite{MR137793}, following Calabi and Vesentini \cite{MR111057, MR111058} who first proved it for $G$ Hermitian
using analytic methods similar to the ones presented in this report.
Mostow used different techniques, essentially studying the dynamics of the action of $\Gamma < G$ on the
ideal boundary of $\mathbf{X} = G/K$. Such techniques were further developed by Margulis to prove spectacular refinements of Mostow rigidity known as \emph{superrigidity} and \emph{arithmeticity} \cite{MR0492072}.
On the other hand, following the work of Siu, the harmonic maps approach was successfully used by Mok \cite{MR1081948} to prove superrigidity for Hermitian Lie groups and by Corlette \cite{MR1147961} for $G = \Sp(n,1)$.
Gromov--Schoen \cite{MR1215595} extended Corlette's argument to groups of rank $1$ over local fields by generalizing the notion of harmonic map to singular targets. Finally, Mok--Siu--Yeung \cite{MR1223224}, Jost--Yau \cite{MR1482040}, and Jost--Zuo \cite{MR1617644} generalized superrigidity for higher rank groups and quasiprojective varieties.

\section{Nonabelian Hodge theory}
\label{sec:NAH}

In this final section, we emphasize the key role of the Siu--Sampson theorem for the nonabelian Hodge correspondence from group representations to Higgs bundles. 
Using an extension of the Eells--Sampson theorem due to Corlette, it will be straighforward to construct a Higgs bundle out of a representation by applying the Siu--Sampson theorem: it is essentially a rephrasing of \autoref{lem:AbelianSub}. The reader in a hurry may jump to \autoref{subsubsec:RepToHiggs}, which explains this point, and skip the rest of the section.

In order to motivate this construction, we attempt a brief introduction to the fascinating subject of nonabelian Hodge theory, which is not a trivial task. 
Thankfully, there are many good references: 
I recommend the phenomenal paper of Simpson \cite{MR1179076} and the more introductory resources 
\cite{MR1157844, MR1379330, MR1492538, MR2359489, Bradlow2012, MR3221295, MR3409775, MR3675465, MR3837868, MR3947036, Brantner}.

\subsection{From classical to nonabelian Hodge theory}

\subsubsection{Classical Hodge theory}
\label{subsubsec:ClassicalHodge}

For a thorough treatment of what follows, refer to \eg{} \cite{MR1924513, MR2451566}.

Let $M$ be a complex manifold. Let us denote $\cA^k = \cA^k(M, \C)$ the space (or sheaf) of $\C$-valued differential $k$-forms
and reserve the notation $\Omega^k \subseteq \cA^k$ for the sheaf of $(k,0)$-forms (see below) that are holomorphic. 
We have seen in \autoref{subsec:ComplexAndKahlerManifolds} that the complexified cotangent space
decomposes as $\upT_{\C}^* M = {\upT^*}^{1,0} M \oplus {\upT^*}^{0,1} M$, this induces the decomposition
$\Lambda^k \upT_{\C}^* M = \bigoplus_{p+q = k} \left(\Lambda^p {\upT_{\C}^*}^{1,0} M\right) \wedge \left(\Lambda^q {\upT_{\C}^*}^{0,1} M\right)$
and accordingly $\cA^k = \bigoplus \cA^{p,q}$. The exterior derivative $\upd \colon \cA^k \to \cA^{k+1}$
also splits as $\upd = \del + \delbar$, where $\del \colon \cA^{p,q} \to \cA^{p+1,q}$ and $\delbar \colon \cA^{p,q} \to \cA^{p,q+1}$.
The fact that $\upd^2 = 0$ implies that $\delbar^2 = 0$, so one can consider the \emph{Dolbeault complex} $(\cA^{p, \bigcdot}, \delbar)$.
The $\delbar$-Poincaré lemma ensures that it is an acyclic resolution of the holomorphic sheaf 
$\Omega^{p} \subseteq \cA^{p,0}$, so that it computes its cohomology: we have $\operatorname{H}^q(\Omega^p) \approx \operatorname{H}^{p,q}_{\textrm{Dol}}$. 
This is the \emph{Dolbeault isomorphism}, analogous to the de Rham isomorphism $\operatorname{H}^k(M, \R) \approx \operatorname{H}_{\textrm{dR}}^k(M,\R)$
\footnote{On the left-hand side of these isomorphisms, we are taking the sheaf cohomology. It can be identified to the \v{C}ech cohomology, as a consequence of manifolds
being assumed paracompact and Hausdorff.}.

Assume that $M$ has a compatible Riemannian metric, \ie{} is Hermitian. We have seen in \autoref{subsec:dfHarmonic} that if $M$ is closed, the de 
Rham cohomology $\operatorname{H}_{\text{dR}}^k(M, \C)$ is isomorphic to the space of harmonic $k$-forms $\cH^k(M,\C)$. Does the Laplacian $\Delta = \upd \upd^* + \upd^* \upd$
respect the decomposition  $\cA^k = \bigoplus \cA^{p,q}$, yielding a refinement $\cH^k(M,\C) = \bigoplus \cH^{p,q}(M,\C)$? The answer is yes provided $M$ is Kähler. Indeed, it is a consequence of the 
\emph{Kähler identities} (see \cite{MR1924513}) that $\Delta = 2\Delta_{\delbar}$, where $\Delta_{\delbar} \coloneqq \delbar \delbar^* + \delbar^* \delbar$.
It follows that $\Delta$ preserves bidegree. Furthermore, we obviously have $\cH^{p,q}(M,\C) = \cH^{p,q}_{\Delta_{\delbar}}(M,\C)$,
and since the operator $\Delta_{\delbar}$ is self-adjoint and elliptic, the same proof as in \autoref{subsec:dfHarmonic}
shows that $\operatorname{H}^{p,q}_{\textrm{Dol}}(M, \C) = \cH^{p,q}_{\Delta_{\delbar}}(M,\C)$.

In summary, if $M$ is compact Kähler, then we have the decomposition $\cH^k(M,\C) = \bigoplus \cH^{p,q}(M,\C)$, which can canonically be rewritten 
\begin{equation} \label{eq:HodgeK}
\operatorname{H}^k(M,\C) \approx \bigoplus_{p+q = k} \operatorname{H}^{p,q}_{\textrm{Dol}}(M,\C) \approx \bigoplus_{p+q = k} \operatorname{H}^{q}(M,\Omega^p)\,.
\end{equation}

\begin{remark}
Remarkably, the isomorphisms in \eqref{eq:HodgeK} turn out to be independent of the Kähler metric. This can be shown by equating
the Dolbeault and the Bott-Chern cohomology using the $\del \delbar$-lemma.
\end{remark}

\begin{remark} \label{rem:HodgeStructure}
 In algebraic geometry, the decomposition of a complex vector space as $V = \bigoplus_{p+q = k} V^{p,q}$ with $V^{q,p} = \overline{V^{p,q}}$
is called a \emph{Hodge structure} of weight $k$\footnote{Typically, an integral structure is also required, \ie{} a lattice $V_\Z$ such that $V = V_\Z \otimes \C$. We ignore that part of the definition, but need at least a real structure so that complex conjugation is well-defined.}. Thus the cohomology of a compact Kähler manifold has a Hodge structure. A Hodge structure of weight $k$ is equivalent to a $\C^*$-action on $V$ such that any $\lambda \in \R^*$ acts by scalar multiplication by $\lambda^k$. 
Indeed, $V^{p,q}$ can be recovered as the subspace where any $\lambda \in \C^*$ acts by scalar multiplication by $\lambda^p \bar{\lambda}^q$.
For more details, refer to \cite[Chap.\ 15]{MR3727160}.
\end{remark}

\subsubsection{Extension to flat Hermitian bundles}
\label{subsubsec:FlatHermitian}

Let $M$ be a compact Kähler manifold.
It is straightforward to generalize the discussion of \autoref{subsubsec:ClassicalHodge} to differential forms with values in a complex vector bundle $E$, 
equipped with a \emph{flat} complex connection $\nabla$ and a compatible Hermitian metric $h$. We already mentioned in \autoref{subsec:dfHarmonic} that the de Rham cohomology and the Hodge isomorphism extend. It remains to define the Dolbeault cohomology, which is done by decomposing $\upd_\nabla$ into types as 
$\upd_\nabla = \del_\nabla + \delbar_\nabla$. We have $\delbar_\nabla^2 = 0$, so we can define the Dolbeault complex, \etc{} The Kähler identities extend to this situation, and similarly to the case $E = \C$ one proves
\begin{equation} \label{eq:HodgeKE}
\operatorname{H}^k(M,E) \approx \bigoplus_{p+q = k} \operatorname{H}^{p,q}_{\textrm{Dol}}(M,E) \approx \bigoplus_{p+q = k} \operatorname{H}^{q}(M,\Omega^p(E))\,.
\end{equation}

It is well-known that the data of an operator $\delbar_E \colon \cA^{p,q}(M,E) \to \cA^{p,q+1}(M, E)$ satisfying the appropriate Leibniz rule and such that $\delbar_E^2 = 0$ 
---called a \emph{Dolbeault operator}---is equivalent to a holomorphic structure in $E$, where $s \in \cA^{0}(E)$ is holomorphic if and only if $\delbar_E s = 0$\footnote{The holomorphic structure associated to a Dolbeault operator is sometimes called the \emph{Koszul--Malgrange} holomorphic structure. Its existence is a consequence of the Newlander-Nirenberg theorem.}. In the presence of a holomorphic structure and a Hermitian metric, it is easy to show that there exists a unique compatible connection $\nabla$ such that $\delbar_E = \delbar_\nabla$, called the \emph{Chern connection}. 
Similary to the case $E = \C$, the Hodge decomposition \eqref{eq:HodgeKE} turns out to be independent of both
the choice of Kähler metric on $M$ and the Hermitian metric on $E$: it only depends on the holomorphic structure on $E$, as long as there exists a Hermitian metric
whose Chern connection is flat---called a \emph{Hermite--Einstein metric}. 
This condition is not trivial, it implies first of all that $E$ has vanishing Chern classes.                                                                                                                                                                                                                                                                                                                                                                                                                                                                                                                                                                                                                                                                                                                                                                                                                                                                                                                                                                                                                                                                                                                                                                                                                                                                                                                                                                                                                                                                                                                                                                                                                                                                                                                                                                                                                                                                                                                                                                                                                                                                                                                                                                                                                                                                                                                                                                                                                                                                                                                                                                                                                                                                                                                                                                                                                                                                                                                                                                                                                                                                                                                                                                                                                                                                                                                                                                                                                                                                                                                                                                                                                                                                                                                                                                                                                                                                                                                                                                                                                                                                                                                 Conversely, such a vector bundle admits a Hermite--Einstein metric if and only if it is \emph{polystable}\footnote{A vector bundle $E$ of degree $0$ is called \emph{stable} if any subbundle has negative degree, \emph{semistable} if any subbundle has nonpositive degree, and \emph{polystable} if it is a direct sum of stable bundles of degree $0$.}. This result is known as the \emph{Hitchin--Kobayashi correspondence}; when $\dim_\C M = 1$ it is the famous theorem of Narasimhan–Seshadri \cite{MR184252}, generalized to higher dimensions by Donaldson \cite{MR765366} and
Uhlenbeck--Yau \cite{MR861491}.

A flat structure on a complex vector bundle is equivalent to a locally constant sheaf (the sheaf of flat sections), \ie{} a \emph{local system}. Alternatively, it is given by a group homomorphism $\rho \colon \pi_1 M \to \GL(n, \C)$ where $n$ is the rank of $E$, which is the holonomy of the flat connection $\nabla$\footnote{In the algebraic setting,
which allows singularities, the equivalence between flat bundles, local systems, and representations is known as the \emph{Riemann-Hilbert correspondence}.}. In our situation, 
$\nabla$ preserving a metric means that its holonomy is unitary, \ie{} $\rho \colon \pi_1 M \to \U(n)$. All in all, the Hitchin--Kobayashi correspondence
yields a bijection between conjugacy classes of unitary representations  and isomorphism classes of polystable holomorphic bundles with vanishing Chern classes.

\subsubsection{Extension to nonabelian groups}
\label{subsubsec:ExtensionNAGroups}

My humble takeaway is that nonabelian Hodge theory can be viewed as an extension of classical Hodge theory for the weight $k=1$ at two different---yet related---levels:
\begin{enumerate}[(i)]
 \item \label{item:NAi} At the ``macroscopic'' level: The classical case gives a Hodge structure on $\operatorname{H}^1(M, \C)$.
 We now seek a Hodge structure on $\operatorname{H}^1(M, G)$ when $G = \GL(n, \C)$ or a more general reductive group.
 \item \label{item:NAii} At the ``microscopic'' level: we saw that Hodge theory works for the cohomology $\operatorname{H}^1(M,E)$ with coefficients in a polystable holomorphic vector bundle $E$, \ie{} correspondi to a unitary representation $\rho \colon \pi_1 M \to \U(n)$. We seek a generalization for arbitrary flat bundles, 
 \ie{} for $\rho \colon \pi_1 M \to \GL(n,\C)$.
\end{enumerate}

\medskip

We now discuss this in a bit more detail, starting with \ref{item:NAi}. The cohomology $\operatorname{H}^k(M, \cF)$ is generally defined for a sheaf $\cF$ of \emph{abelian} groups: this ensures that quotient groups are well-defined. When $\cF$ is a nonabelian sheaf, the $k=1$ cohomology $\operatorname{H}^1(M, \cF)$ is still well-defined as a set. Recall that for $\cF = \C$, the Hodge decomposition was
\begin{equation} \label{eq:HodgeK1}
\operatorname{H}^1(M,\C) 
\approx \operatorname{H}^{0,1}_{\textrm{Dol}}(M,\C) \oplus \operatorname{H}^{1,0}_{\textrm{Dol}}(M,\C) 
\approx \operatorname{H}^{1}(M, \cO) \oplus \operatorname{H}^{0}(M,\Omega^1)\,.
\end{equation}
Consider now $G = \GL(n,\C)$, denote also by $G$ the constant sheaf over $M$, and $\cG$ its sheaf of holomorphic sections.
The cohomology $\operatorname{H}^{1}(M, G)$ can be identified to the space $\Hom(\pi_1 M, G)/G$ of group homomorphisms up to conjugation, 
or to the space of $G$-local systems over $M$. The analogue of
$\operatorname{H}^{1}(M, \cO)$ is $\operatorname{H}^{1}(M, \cG)$, which parametrizes holomorphic bundles of rank $n$ over $M$ up to isomorphism (this is easy to see with \v{C}ech cohomology). We shall see that the second piece $\operatorname{H}^{0}(M,\Omega^1)$ will be replaced by $\operatorname{H}^{0}(M,\End(\cE) \otimes \Omega^1 )$, where $\cE$ is the holomorphic vector bundle under consideration. Note that, since the second piece (called the \emph{Higgs field}) is defined relative to the first (the holomorphic bundle), the ``Hodge decomposition'' of $\operatorname{H}^{1}(M, G)$ is a twisted sum rather than a direct sum.

Now let us expand on \ref{item:NAii}. We have seen that Hodge theory applies in flat bundles that preserve a Hermitian metric, equivalently, to holomorphic bundles with vanishing Chern classes that admit a Hermite--Einstein metric. Simpson extended the Kähler identities to flat bundles that admit a \emph{harmonic} metric (see \autoref{subsec:NAH}), and showed that they correspond to polystable Higgs bundles with vanishing Chern classes, extending the Hitchin--Kobayashi correspondence. On the other hand, Corlette showed that harmonic flat bundles correspond to reductive representations $\rho \colon \pi_1 M \to \GL(n,\C)$
\footnote{\label{foot:Reductive}A representation $\rho \colon \pi_1 M \to G = \GL(n,\C)$ is called \emph{reductive} (or \emph{completely reducible}, or \emph{polystable}) if its action on $\C^n$ is completely reducible. More generally, if $G$ is any algebraic group, $\Gamma \coloneqq \rho(\pi_1 M) < G$ is called completely reducible if the identity component of its Zariski closure is a reductive subgroup, \ie{} with trivial unipotent radical. Equivalently, for every parabolic subgroup $P<G$ containing $\Gamma$, there is a Levi factor of $P$ containing $\Gamma$ \cite{MR2931326}.
% If $\Gamma$ is not contained in any parabolic subgroup, $\rho$ is called \emph{irreducible} or \emph{stable}.
}. Referring to these two correspondences, Simpson writes (\cite[below Thm.\ 1]{MR1159261}): 
\begin{quote}
The set of flat bundles is analogous to the abelian de Rham cohomology, while
the set of Higgs bundles is analogous to the abelian Dolbeault cohomology,
$\operatorname{H}^1(\cO) \oplus \operatorname{H}^0(\Omega^1)$. The first two parts of Theorem 1
may be interpreted as giving harmonic representatives for certain nonabelian de Rham and Dolbeault cohomology classes. 
The fact that the notion of harmonic representative (harmonic
bundle) is the same in both cases, is the analogue of the abelian Kähler identity $\Delta_{\upd} = 2 \Delta_{\delbar}$.
\end{quote}

% \begin{remark}
% The theory of Higgs bundles was initiated by Hitchin in the monumental paper \cite{MR887284}, in which he introduced Higgs bundles over a Riemann surface, constructed their moduli space, and proved many of its striking features---mostly for rank $2$ bundles. One of his main results is the generalization of the Narasimhan--Seshadri theorem to Higgs bundles. Donaldson added a key result in \cite{MR887285}, an extension of the Eells--Sampson theorem to the equivariant setting, which completed the correspondence between group representations and Higgs bundles in Hitchin's setting. Simpson \cite{MR944577} and Corlette \cite{MR965220} extended the results of Hitchin and Donaldson to Higgs bundles (of any rank) over a compact Kähler manifold. Simpson went on to masterfully develop nonabelian Hodge theory in a series of papers \cite{MR1159261, MR1179076, SimpsonModuli, MR1492538}. Of course, this rough historical account is very incomplete!
% \end{remark}

\subsection{The nonabelian Hodge correspondence}
\label{subsec:NAH}

\subsubsection{From representations to Higgs bundles}
\label{subsubsec:RepToHiggs}

Let $M$ be a compact Kähler and $G$ be a reductive complex algebraic group. For simplicity, we first consider $G = \GL(n, \C)$.

Consider a flat rank $n$ vector bundle $(E, \nabla)$ over $M$. Recall that this is equivalent to a $G$-local system, or to a representation $\rho \colon \pi_1 M \to G$. Any Hermitian metric $h$ in $E$ induces a ``unitary $+$ self-adjoint'' decomposition $\nabla = \nabla_h + \psi_h$, where $\nabla_h$ is a connection preserving $h$ and $\psi_h \in \cA^1(M, \operatorname{Sym}_h(E))$ is a $1$-form with values in $h$-self-adjoint endomorphisms. The choice of the Hermitian metric $h$ is equivalent to a reduction of the structure group of $E$ from $G$ to $K = \U(n)$, given by a smooth $\rho$-equivariant\footnote{By definition, $f \colon \tilde{M} \to G/K$ is \emph{$\rho$-equivariant} if $f(\gamma \cdot x) = \rho(\gamma) \cdot f(x)$ for all $x \in M$ and $\gamma \in \pi_1M$, where $\pi_1 M$ acts by deck transformations on $\tilde{M}$ and $G$ acts by left translations on $G/K$.} map $f \colon \tilde{M} \to \mathbf{X} = G/K$ called the \emph{classifying map} of the metric.
Its differential $\upd f$ can be seen as a $1$-form with value in $f^* (\upT \mathbf{X})$, and the latter bundle can be identified $\operatorname{Sym}_h(E)$. One readily checks that $\psi_h = \upd f$ under this identification,
while $\nabla_h$ is essentially the pullback of the Levi-Civita connection of $\mathbf{X}$.

\begin{definition} \label{def:HarmonicFlatBundle}
 A Hermitian metric $h$ in a flat vector bundle $(E, \nabla)$ is \emph{harmonic} if it satisfies one of the equivalent conditions:
 \begin{enumerate}[(i)]
  \item $h$ is a critical point of the energy functional $\mathbf{E}(h) = \frac{1}{2} \int_M \Vert \psi_h \Vert^2 \upd v_M$.
  \item $\upd_{\nabla_h}^* \psi_h = 0$.
  \item \label{item:HarmonicFlatBundleiii} The classifying map $f \colon \tilde{M} \to \mathbf{X} = G/K$ is harmonic.
 \end{enumerate}
\end{definition}
The equivalence of the conditions in \autoref{def:HarmonicFlatBundle} is exactly \autoref{thm:CharacHarmonic}.

The third characterization \ref{item:HarmonicFlatBundleiii} implies that the question of the existence and uniqueness of a harmonic metric on $(E, \nabla)$ is the same as that of $\rho$-equivariant harmonic map $f \colon \tilde{M} \to \mathbf{X}$. 
More generally, one can study $\rho$-equivariant harmonic maps $f \colon \tilde{M} \to N$ where $M$ and $N$ are Riemannian manifolds, with $\rho \colon \pi_1 M \to \Isom(N)$. This is an extension of the classical situation studied by Eells--Sampson. When $N = \mathbf{X}$ is a symmetric space of noncompact type, Corlette proved:
\begin{theorem}[Corlette \cite{MR965220}] \label{thm:Corlette}
Let $M$ be a compact Riemannian manifold, let $\mathbf{X} = G/K$ be a symmetric space of noncompact type, and let $\rho \colon \pi_1 M \to G$ be a group homomorphism. There exists a $\rho$-equivariant harmonic map $f \colon \tilde{M} \to \mathbf{X}$ if and only if $\rho$ is reductive.
\end{theorem}

\begin{remark}
 Like Corlette, we stated \autoref{thm:Corlette} for a symmetric space of noncompact type, \ie{} for a semisimple Lie group $G$, but it holds more generally for a reductive group such as $G = \GL(n, \C)$. The same goes for most of \autoref{sec:SymmetricSpaces} and \autoref{lem:SSNAH} below.
\end{remark}

\begin{remark}
Labourie \cite{MR1049845} gave a more general version of \autoref{thm:Corlette}, replacing $\mathbf{X}$ by any simply-connected Riemannian manifold $N$ of nonpositive sectional curvature ``without flat half-strips''. His short proof is an extension of the heat flow method of Eells--Sampson described in \autoref{subsubsec:HeatFlow}.
\end{remark}

\begin{remark}
An application of the classical uniqueness result for harmonic maps (due to Hartman \cite{MR0214004} and Al'ber \cite{MR0230254}) shows that in \autoref{thm:Corlette}, when a $\rho$-equivariant harmonic map $f$ exists, it is necessarily unique up to post-composition with elements of the centralizer of $\rho(\pi_1(M)) < G$. 
\end{remark}

In the language of bundles, Corlette's theorem says that a flat bundle $(E, \nabla)$ admits a (essentially unique) harmonic metric if and only if it is semisimple. We are now halfway in the nonabelian Hodge correspondence from (reductive) representations, \ie{} (semisimple) flat bundles, to Higgs bundles:
\begin{equation} \label{eq:NAH1}
\begin{tikzcd} 
 \left\{\text{flat bundles}\right\} \ar[r] & \left\{\text{harmonic bundles}\right\} \ar[r, dashed, "?"]  & \left\{\text{Higgs bundles}\right\}
\end{tikzcd}
\end{equation}
The second step, from harmonic bundles to Higgs bundles, is where the theorem of Siu--Sampson is critical. Specifically, we have the key lemma, which is little more than a rephrasing of \autoref{lem:AbelianSub}:
\begin{lemma} \label{lem:SSNAH}
Let $f \colon M \to \mathbf{X}$, where $M$ is compact Kähler and $\mathbf{X} = G/K$ is symmetric  of noncompact type. Denote by $\upd_\nabla$ the exterior covariant derivative in $\Omega^{\bullet}(M, f^* \upT \mathbf{X})$ as in \autoref{subsec:dfHarmonic}. If $f$ is harmonic, then:
\begin{align}
 \left(\delbar_\nabla\right)^2 &= 0 \label{eq:PHS1} \\
\delbar_\nabla {\upd}^{1,0} f  &= 0 \label{eq:PHS2} \\
 \left[\upd^{1,0}f, \upd^{1,0}f\right] &=0 \label{eq:PHS3} \,.
\end{align}
\end{lemma}
\begin{remark}
The expression $[\upd^{1,0}f, \upd^{1,0}f]$ is an element of $\cA^2(M, f^*\upT_\C \mathbf{X})$, taking the exterior product of one-forms with the Lie bracket in $\upT_\C X \approx \mathfrak{g}$ as coefficient pairing. When $G = \GL(n,\C)$, we saw that $\upd f = \psi_h$ is a one-form with values in $\operatorname{Sym}_h(E)$, and the Lie bracket is the commutator of endomorphisms.
\end{remark}

\begin{proof}[Proof of \autoref{lem:SSNAH}]
\autoref{lem:AbelianSub} immediately gives \eqref{eq:PHS2}, which is a characterization of pluriharmonicity (\autoref{prop:CharacterizationsPluriharmonic}), and \eqref{eq:PHS3}, which is the fact that $\mathfrak{a}_x \coloneqq \upd f\left(\upT_x^{1,0} M\right)$ is an abelian subalgebra of $\mathfrak{g}$ for all $x \in M$. As for \eqref{eq:PHS1}, it is the condition that 
$R^\mathbf{X}(\upd f (X), \upd f(Y)) = 0$ for all $X, Y \in \upT^{0,1} M$, in other words $\left[\upd^{0,1}f, \upd^{0,1}f\right] =0$, which is easily derived from $\left[\upd^{1,0}f, \upd^{1,0}f\right] =0$ by conjugation.
\end{proof}

\begin{remark}
Consider the vector bundle $f^* \upT_\C N$ over $M$. If $f$ is harmonic, by \eqref{eq:PHS1}, $\delbar_\nabla$ is a Dolbeault operator, giving
it a holomorphic structure. \eqref{eq:PHS2} then says that $\upd^{1,0} f$ is a holomorphic $(1,0)$-form.
\end{remark}

\begin{theorem}[Application of the Siu--Sampson theorem] \label{thm:SSApp}
Let $(E, \nabla)$ be a flat complex vector bundle. If $h$ is a harmonic metric on $(E, \nabla)$, then $\delbar_E \coloneqq \delbar_{\nabla_h}$ and $\varphi \coloneqq \left(\psi_h\right)^{1,0}$ satisfy:
\begin{align} \label{eq:HiggsEq}
 \delbar_{E}^2 &= 0\\
 \delbar_{E} \varphi  &= 0\\
 [\varphi , \varphi ] &=0\,.
\end{align}
Moreover, we have \emph{Hitchin's equation}\footnote{Hitchin calls this equation together with $\delbar_{E} \varphi = 0$ the \emph{self-duality equations} \cite{MR887284}.}:
\begin{equation} \label{eq:SelfDuality}
F^{\nabla_h} + \left[\varphi, \varphi^{*_h} \right] = 0\,.
\end{equation}
\end{theorem}

\begin{proof}
We have seen that under appropriate identifications, $\psi_h = \upd f$ and $\nabla_h$ is the pullback of the Levi-Civita connection of $\mathbf{X}$.
The first part of the theorem is then simply a rephrasing of \autoref{lem:SSNAH}.
Finally, Hitchin's equation amounts to the flatness of $\nabla$. Indeed, first write $F^\nabla = F^{\nabla_h} + \nabla_h \psi_h + \psi_h \wedge \psi_h =  0$. Decomposing into $\pm$-$h$-self-adjoint components:
 \begin{align} \label{eq:Flatness1}
 F^{\nabla_h} + \psi_h \wedge \psi_h &= 0\\
 \nabla_h \psi_h &= 0\,.
\end{align}
The second equation can be rewritten $\upd_{\nabla_h} \upd f = 0$, and we have seen in \autoref{prop:dfclosed} that this is actually automatic. The first equation can be rewritten $F^{\nabla_h} + \frac{1}{2} [\psi_h , \psi_h] = 0$. Since $\psi_h$ is $h$-self-adjoint, one must have $\psi^{0,1} = \left(\psi^{1,0}\right)^{*_h}$, \ie{} $\psi = \varphi + \varphi^{*_h}$. Since $[\varphi , \varphi ] = \left[\varphi^{*_h} , \varphi^{*_h} \right] = 0$, the conclusion follows.
\end{proof}

% According to \autoref{thm:SSApp}, if $(E, \nabla, h)$ is a flat harmonic bundle, then $\delbar_E \coloneqq \delbar_{\nabla_h}$ is a Dolbeault operator
% in $E$, hence defines a holomorphic structure in $E$, and $\varphi \coloneqq \left(\psi_h\right)^{1,0}$ is a holomorphic $(1,0)$-form with values in $\End(E)$, \ie{} a holomorphic section of $\End(E) \otimes \Omega^1$, which satisfies $[\varphi, \varphi] = 0$.

\begin{definition}
 A \emph{Higgs bundle} is a pair $(\cE, \varphi)$, where $\cE$ is a holomorphic vector bundle and $\varphi$ is a holomorphic section of $\End \cE \otimes \Omega^1$ such that $[\varphi, \varphi] = 0$.
\end{definition}

% \begin{remark}
% Higgs bundles were introduced by Hitchin \cite{MR887284} (originally for $n=2$ and $\dim_\C X = 1$). Hitchin coined the term \emph{Higgs field} for $\varphi$, and Simpson called the pair $(E, \varphi)$ \emph{Higgs bundle}.
% \end{remark}

\autoref{thm:SSApp} thus says that if $(E, \nabla, h)$ is a flat harmonic bundle, then $(\cE, \varphi)$
is a Higgs bundle, where $\cE \coloneqq (E, \delbar_E)$.  This complete our description of the nonabelian Hodge correspondence from representations to Higgs bundles.

% \begin{remark}
%  As a sanity check, I recommend going over this subsection again and examine the special case where one starts with a unitary representation $\rho \colon \pi_1 M \to \U(n)$: what are $h$, $f$, $\nabla_h$, $\psi_h$, $\delbar_E$, and $\varphi$ in this case? Is this situation consistent with 
%  \autoref{subsubsec:FlatHermitian}?
% \end{remark}

\subsubsection{From Higgs bundles to representations}
\label{subsubsec:HiggsToRep}

In the previous subsection, we described how to produce a harmonic flat bundle out of a representation, and then a Higgs bundle. This construction also yielded a metric $h$ satisfying Hitchin's equation \eqref{eq:SelfDuality}. Note that, forgetting the flat connection $\nabla$, Hitchin's equation still makes sense given the Higgs bundle $(\cE, \varphi)$: one should then understand $\nabla_h$ as the Chern connection of $h$ in $\cE$.
\begin{definition}
A Hermitian metric $h$ in a Higgs bundle $(\cE, \varphi)$ with vanishing Chern classes is called \emph{Hermitian--Yang--Mills} if it satisfies Hitchin's equation \eqref{eq:SelfDuality}.
\end{definition}

We saw that if $(E, \nabla, h)$ is a flat harmonic bundle, then $(\cE, \varphi, h)$ is a Hermitian--Yang--Mills Higgs bundle, with
$\delbar_E \coloneqq \delbar_{\nabla_h}$ and $\varphi \coloneqq \left(\psi_h\right)^{1,0}$, where $\nabla = \nabla_h + \psi_h$ is the unitary $+$ self-adjoint decomposition of $\nabla$. Conversely, $(E, \nabla, h)$ is easily reconstructed from $(\cE, \varphi, h)$: put $\nabla = \nabla_h + \psi_h$, where $\nabla_h$ is the Chern connection of $h$ in $\cE$, and $\psi_h \coloneqq \varphi + \varphi^{*_h}$. It is then straightforward to check that $\nabla$ is flat and $h$ is harmonic, using the equations $F^{\nabla_h} + \left[\varphi, \varphi^{*_h} \right] = 0$, $\delbar_{\nabla_h} \varphi = 0$, and $[\varphi, \varphi] = 0$. Let us put this on record:
\begin{proposition}
Let $E \to M$ be a complex vector bundle with vanishing Chern classes. There is a $1$:$1$ correspondence between harmonic flat bundles $(E, \nabla, h)$ and Hermitian--Yang--Mills Higgs bundles $(\cE, \varphi, h)$.
\end{proposition}

\begin{remark}
Instead of \emph{Hermitian--Yang--Mills} metrics, \cite{Brantner} says \emph{Higgs--Hermite--Einstein}  to emphasize that they generalize Hermite--Einstein metrics. For Higgs bundles with vanishing Chern classes, it is sensible to call them \emph{harmonic}. The coincidence of harmonicity for flat bundles and Higgs bundles is the analogue of the Kähler identity $\Delta_{\upd} = 2\Delta_{\delbar}$ that we referred to in \autoref{subsubsec:ExtensionNAGroups} in the words of Simpson.
\end{remark}

To complete the nonabelian Hodge correspondence, we need the fundamental result of Hitchin \cite{MR887284} (for the rank $2$ case over a Riemann surface) and Simpson \cite{MR944577} (for the general case), generalizing the Hitchin-Kobayashi correspondence:
\begin{theorem}[Simpson] \label{thm:Simpson}
A Higgs bundle admits a Hermitian--Yang--Mills metric if and only if it is polystable\footnote{The notion of stability for Higgs bundles is a natural 
refinement of the notion for holomorphic bundles: in the definition, only considers Higgs subbundles, \ie{} holomorphic subbundles preserved by the Higgs field.}. 
\end{theorem}
The proof of \autoref{thm:Simpson} relies on geometric analysis; it consists in minimizing a Yang--Mills functional. We shall not discuss any details and refer to the excellent \cite{MR3675465} instead.

Let us recap with a diagram slightly more detailed than \eqref{eq:NAH1}:

{\small
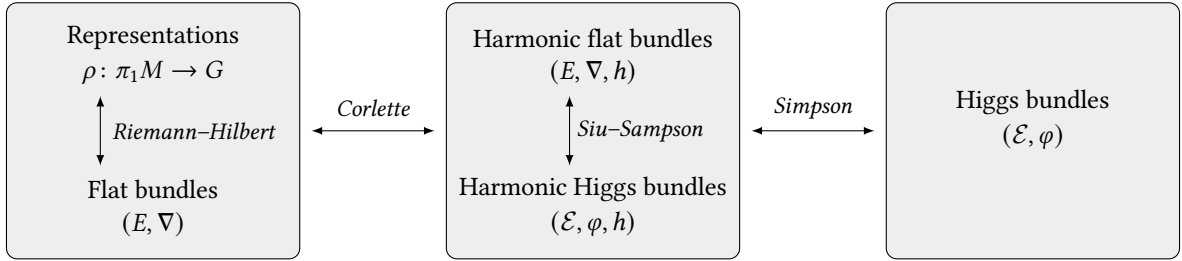
\begin{equation}
\noindent \begin{tikzpicture}[scale = 0.98\textwidth/(10*(72.27/2.54))] % NB: 10 = width of picture, 72.27/2.54 = pts per cm

\draw [rounded corners, fill=black, fill opacity=0.07] (0,0.1) rectangle (2.5, 2.3) ;
\draw (1.25, 2.0) node {Representations} ;
\draw (1.25, 1.7) node {$\rho \colon \pi_1 M \to G$} ;
\draw[<->,>=latex] (0.8, 1.5) to (0.8, 0.9);
\draw (1.6, 1.2) node {\footnotesize \emph{Riemann--Hilbert}} ;
\draw (1.25, 0.7) node {Flat bundles} ;
\draw (1.25, 0.4) node {$(E, \nabla)$} ;

\draw[<->,>=latex] (2.6, 1.2) to (3.65, 1.2);
\draw (3.125, 1.4) node {\footnotesize \emph{Corlette}} ;
  
\draw [rounded corners, fill=black, fill opacity=0.07] (3.75,0.1) rectangle (6.25, 2.3) ;
\draw (5.0, 2.0) node {Harmonic flat bundles} ;
\draw (5.0, 1.7) node {$(E, \nabla, h)$} ;
\draw[<->,>=latex] (4.8, 1.5) to (4.8, 0.9);
\draw (5.4, 1.2) node {\footnotesize \emph{Siu--Sampson}} ;
\draw (5.0, 0.7) node {Harmonic Higgs bundles} ;
\draw (5.0, 0.4) node {$(\cE, \varphi, h)$} ;

\draw[<->,>=latex] (6.35, 1.2) to (7.4, 1.2);
\draw (6.875, 1.4) node {\footnotesize \emph{Simpson}} ;
% \draw (6.875, 1.0) node {\footnotesize \emph{(Hitchin--Koba)}} ;
   
\draw [rounded corners, fill=black, fill opacity=0.07] (7.5,0.1) rectangle (10, 2.3) ;
\draw (8.75, 1.45) node {Higgs bundles} ;
\draw (8.75, 1.15) node {$(\cE, \varphi)$} ;
\end{tikzpicture}
\end{equation}}
% \begin{remark}
%  We should throw some adjectives in the diagram above to make it more accurate: \emph{reductive} representations, \emph{semisimple} flat bundles,
%  \emph{polystable} Higgs bundles \emph{with vanishing Chern classes}.
% \end{remark}

Let us give a cute return on investment from our study of pluriharmonic maps (\autoref{subsec:Pluriharmonic}):
\begin{proposition}
Let $Q \subseteq M$ be a complex submanifold. Any polystable Higgs bundle $(\cE, \varphi)$ on $M$ stays polystable in restriction to $Q$.
\end{proposition}
\begin{proof}
By Simpson's theorem, there exists a harmonic metric $h$ on $(\cE, \varphi)$. Let 
$f \colon \tilde{M} \to G/K$ indicate the associated classifying map. We have seen that $h$ being harmonic amounts to $f$ being harmonic. By the Siu--Sampson theorem, $f$ is pluriharmonic (\autoref{lem:AbelianSub}). It follows that the restriction of $f$ to $\tilde{N}$ is pluriharmonic (\autoref{prop:PropertiesPluriharmonic}). This means that $h$ is still harmonic as a metric on the Higgs bundle $(\cE, \varphi)$ restricted to $N$. By Simpson's theorem, it must be polystable.
\end{proof}

\subsubsection{More general groups}

We now briefly describe how to generalize the nonabelian Hodge correspondence to more general groups than $\GL(n,\C)$.
Refer to \cite{MR3221295} for details and \cite{DupontFB} for background material on fiber bundles. 

Let $G$ be a reductive\footnote{\ie{} with trivial unipotent radical. Equivalently, it admits a faithful semisimple representation \cite{MR3729270}. A connected affine algebraic group over $\C$ is reductive if and only if it has a reductive Lie algebra and its center is of multiplicative type.} %It can also be characterized as the complexification of a compact real algebraic group.} 
complex algebraic group. The Riemann--Hilbert correspondence gives a bijection between group representations $\rho \colon \pi_1 M \to G$, $G$-local systems, and flat principal $G$-bundles.

Let $\rho \colon \pi_1 M \to G$ be a reductive representation. Denote by $E_G$ be the associated flat principal $G$-bundle and 
$\omega \in \cA^1(M, \mathfrak{g})$ its principal connection, which is flat: $\upd \omega + \frac{1}{2}[\omega, \omega] = 0$.
The Cartan decomposition $\mathfrak{g} = \mathfrak{k} \oplus \mathfrak{p}$ splits the connection as $\omega = \omega_\mathfrak{k} + \omega_\mathfrak{p}$. 

\begin{remark}
Since $G$ is complex reductive, it is isomorphic to the complexification of its maximal compact: $G \approx K^\C$ (the converse is also true). In particular, $\mathfrak{p} = i \mathfrak{k}$.
\end{remark}

\autoref{thm:Corlette} gives a $\rho$-equivariant harmonic map $f \colon \tilde{M} \to \mathbf{X} = G/K$. This induces a reduction of the structure group of $E_G$: we have $K$-bundle $E_K$ and an isomorphism of $G$-bundles $\iota \colon E_K \to E_G$. We denote $\iota^* \omega = A + \psi$ the splitting of the flat connection in $E_K$. Now $A \in \cA^1(E_K, \mathfrak{k})$ is a connection on $E_K$, and $\psi \in \cA^1(E_K, \mathfrak{p})$ descends as a one-form on $M$ with values in $E_K(\mathfrak{p})$, the bundle associated to $E_K$ via the isotropy representation $\Ad \colon K \to \GL(\mathfrak{p})$. 

The vanishing of the curvature of $\iota^* \omega$, broken into $\mathfrak{k}$ and $\mathfrak{p}$ components, gives the equations 
\begin{align}
 F^A + \frac{1}{2}[\psi, \psi] &= 0\\
 \upd_A \psi &= 0
\end{align}
(compare with \eqref{eq:Flatness1}), while the harmonicity of $f$ equates to $\upd_A^* \psi = 0$. We are now halfway in the nonabelian Hodge correspondence, having obtained a harmonic bundle (see \eqref{eq:NAH1}).

We can now apply the Siu--Sampson theorem; in fact, \autoref{lem:SSNAH} still holds verbatim.
The $(1,0)$-form $\varphi = \psi^{1,0}$ makes sense as a form with values in 
$E_K(\mathfrak{g})$ (since $\mathfrak{p}^\C = \mathfrak{g}$), but we deliberately extend the structure group from $K$ to $G$---to ``forget the harmonic metric'' in the Higgs bundle. As in \autoref{thm:SSApp}, we obtain a Higgs bundles, according to the definition:
\begin{definition}
 A $G$-Higgs bundle is a pair $(\cE_G, \varphi)$, where $\cE_G$ is a holomorphic principal $G$-bundle
 and $\varphi$ is a holomorphic section of $\cE_G(\mathfrak{g}) \otimes \Omega^1$ such that $[\varphi, \varphi] = 0$.
\end{definition}

\begin{remark}
When defining \emph{real} Higgs bundles, \ie{} when $G$ is a real reductive group, tracing back our discussion shows that the Higgs field should be defined as a holomorphic section of 
$E_{K^\C}(\mathfrak{p}^\C) \otimes \Omega^1$.
\end{remark}

We have successfully constructed a Higgs bundle starting from 
a reductive representation. Moreover, as in \autoref{thm:SSApp}, the flatness of the connection yields Hitchin's equation
\begin{equation} \label{eq:HitchinPrinc}
F^A - [\varphi, \tau(\varphi)] = 0
\end{equation}
where $\tau \colon \mathfrak{g} \to \mathfrak{g}$ is the Cartan involution, extended to $\cE_G(\mathfrak{g}) \otimes \cA^1$ by $\tau(A \otimes \alpha) = \tau(A) \otimes \bar{\alpha}$.

As expected, \autoref{thm:Simpson} also extends to this setting. First notice that a reduction of the structure group of $\cE$ from $G$ to $K$
induces a Chern connection $A$, so that \eqref{eq:HitchinPrinc} makes sense (see \cite{garciaprada2012hitchinkobayashi}).
\begin{theorem}[Simpson]
Let $(\cE, \varphi)$ be a $G$-Higgs bundles with vanishing Chern classes. There exists a reduction of the structure group of $\cE$ from $G$ to $K$
such that \eqref{eq:HitchinPrinc} holds if and only if $(\cE, \varphi)$ is polystable.
\end{theorem}
As before, this theorem gives an inverse of the nonabelian Hodge correspondence from representations to Higgs bundles, and the circle is complete.

\subsubsection{The character variety and the moduli space of Higgs bundles}

At this point we have described the nonabelian Hodge correspondence from group representations to Higgs bundles and vice-versa, but avoided talking about the moduli spaces of such objects. It turns out that their construction works strikingly well. Let us very briefly sketch the essential ideas at play.

There are natural morphisms between the objects under consideration, in particular:
\begin{enumerate}[\textbullet]
 \item Two representations are equivalent if they are conjugated by an element of $G$.
 \item Two Higgs bundles are equivalent if they differ by a gauge transformation\footnote{A gauge transformation
 of a complex vector bundle $E$ is a smooth section of $\GL(E)$. This definition easily generalizes to $G$-bundles.
 We are being sloppy when talking about gauge equivalence between Higgs bundles, because this assumes that they all have the same underlying complex
 vector bundle---which is true up to isomorphism.
 }.
\end{enumerate}
With these isomorphisms, the nonabelian Hodge correspondence is an \emph{equivalence of categories} between reductive representations and polystable Higgs bundles with vanishing Chern classes (\cite{SimpsonModuli}, see \cite{MR2400111} for a nice exposition of the rank one case). 
In particular there is a 1:1 correspondence between the isomorphism classes. That being said, it is possible to considerably strengthen this assertion.

Firstly, Simpson extended the equivalence of categories to \emph{all} representations $\pi_1 M \to G$ and \emph{all semistable} Higgs bundles
with vanishing Chern classes \cite[\S 3]{MR1179076}. Secondly, and more importantly, the quotient sets can beautifully be realized as algebraic or analytic quotients, yielding a very rich structure.

The best scenario is when $M$ and $G$ both have algebraic structures. The sets $\Hom^{\textrm{red}}(\pi_1 M, G) / G$ (reductive representations up to conjugation) and $\operatorname{Higgs}_0^{\textrm{ps}}(M,G) / \cG$ (polystable Higgs bundles with vanishing Chern classes up to equivalence) can then be realized as \emph{GIT quotients}. Consequently, they inherit an algebraic structure, and are the \emph{universal categorical quotients}. The first GIT quotient is straightforward to construct, see below. The second is more involved: refer to \cite{MR1123536, SimpsonModuli, MR1492538}.

\begin{remark}
The standard reference for GIT is \cite{MR1304906}. For our usage, a nice alternative is \cite{MR2408226}. 
I also point to my notes \cite[\S \textbf{B}]{LoustauAsheville} for an easy introduction to symplectic reduction and GIT.
\end{remark}

Let $X$ be an algebraic variety over $\C$ and let $G$ be a complex reductive algebraic group acting on $X$. For simplicity we assume $X$ is affine; this includes $X = \Hom(\pi, G)$ where $\pi$ is any finitely presented group.
Let $\C[X]$ denote the coordinate ring of regular functions on $M$
and $\C[X]^{G} \subseteq \C[X]$ the subalgebra of $G$-invariant functions.
$\C[X]^{G}$ is finitely generated provided $G$ is reductive by Nagata's theorem (answering Hilbert's 14th problem), therefore it is the coordinate ring 
of an affine algebraic set, namely $\Spec \C[X]^{G}$. This algebraic set is the \emph{GIT quotient} of $X$, denoted $X/\!/G$.
As a topological space, $X /\!/G$ is isomorphic to the set $X^{\textrm{ps}}/G$ of orbits of \emph{polystable} points, \ie{} whose orbit is closed, or to the set $X^{\textrm{ss}}/\sim$ of \emph{semistable}\footnote{Let us not give the general definition, but
point out that all points are semistable when $X$ is affine.} points up to equivalence, where two points are equivalent if their orbit closures interset. Indeed, it is a general fact that %$X \to X /\!/G$ is a \emph{good quotient}, which implies that 
any semistable orbit contains a unique polystable orbit. 
% $X /\!/G$ is the largest Hausdorff quotient of $X^{\textrm{ss}}/G$.

We call the GIT quotients $\cM_{\textrm{B}} \coloneqq \Hom(\pi_1 M, G) /\!/ G$ and 
$\cM_{\textrm{Dol}} \coloneqq \operatorname{Higgs}_0(M,G) /\!/ \cG$ the \emph{$G$-character variety} or \emph{Betti moduli space}, and the \emph{moduli space of $G$-Higgs bundles (with vanishing Chern classes)} or \emph{Dolbeault moduli space}. Both are finite-dimensional quasiprojective varieties.
One can check that polystability in the sense of GIT coincides with reductivity for representations and polystability for Higgs bundles. 
Since the nonabelian Hodge correspondence induces a bijection between the sets $\Hom^\textrm{red}(\pi_1 M, G) / G$ and 
$\operatorname{Higgs}_0^\textrm{ps}(M,G) / \cG$, we obtain:
\begin{theorem}
 The nonabelian Hodge correspondence induces a bijection $\cM_{\textrm{B}} \isoto \cM_{\textrm{Dol}}$.
\end{theorem}
We shall see that this bijection is a real-analytic map between varieties, but not complex-analytic.

In addition to giving a natural framework for the nonabelian Hodge correspondence, the benefit of constructing the moduli spaces
is to provide additional structure and features. For instance, we have seen that GIT quotients provide an algebraic structure. 
Another remarkable feauture is their relation to symplectic reduction (this is essentially the Kempf--Ness theorem, see \cite[Chap.\ 8]{MR1304906}). 

Goldman \cite{MR762512}, following Atiyah--Bott \cite{MR702806},
adapted the technique of symplectic reduction to the infinite-dimensional space of connections in a vector bundle, interpreting the curvature of a connection as a moment map for the action of the gauge group. The main result of \cite{MR762512} is the existence of a natural complex symplectic structure
in the character variety $\cM_{\textrm{B}}$ when $\dim_\C M = 1$. 

Hitchin \cite{MR887284} used similar techniques to construct the moduli space of Higgs bundles $\cM_{\textrm{Dol}}$ as a hyper-Kähler quotient, 
and showed that the nonabelian Hodge correspondence $\cM_{\textrm{B}} \isoto \cM_{\textrm{Dol}}$ is complex-analytic with respect to one of the 
complex structures.
Deligne (unpublished, see \cite{MR1492538}) realized that the twistor space of this hyper-Kähler structure can elegantly be realized as the moduli space
of \emph{$\lambda$-connections}, which Simpson calls \emph{Hodge moduli space} $\cM_{\textrm{Hod}}$. In this picture, $\cM_{\textrm{B}}$ and $\cM_{\textrm{Dol}}$ are two special fibers of the holomorphic fiber bundle $\cM_{\textrm{Hod}} \to \CP^1$.

We finally mention two other remarkable features of the moduli space of Higgs bundles $\cM_{\textrm{Dol}}$, both going back to Hitchin \cite{MR887284}.
Via the nonabelian Hodge correspondence, they also describe fascinating features of the character variety:
\begin{enumerate}[\textbullet]
 \item \emph{The $\C^*$-action.} There is an action of $\C^*$ on Higgs bundles by $\lambda \cdot (\cE, \varphi) = (\cE, \lambda \varphi)$, inducing
 an action on $\cM_{\textrm{Dol}}$. When transported
 to the character variety , this $\C^*$-action should be interpreted as a Hodge structure
 on $\operatorname{H}^1(M,G) = \Hom(\pi_1 M, G)/G$ that generalizes the Hodge decomposition of $\operatorname{H}^1(M, \C)$ (see \autoref{subsubsec:ExtensionNAGroups} and \autoref{rem:HodgeStructure}).
 Simpson  additionally showed  that the fixed points of the action correspond to \emph{variations of Hodge structure} \cite{MR1179076}.
 \item \emph{The Hitchin fibration.} For simplicity, we take $G = \GL(n, \C)$. 
 The coefficients of the characteristic polynomial of the Higgs field defines a holomorphic fibration
 \begin{equation}
  \cM_{\textrm{Dol}} \to \bigoplus_{j=0}^n \operatorname{H}^0\big(M, \left(\Omega^1\right)^{\otimes j}\big)\,.
 \end{equation}
 The study of this map and its fibers by various authors has produced an abundance of interesting results concerning both Higgs bundles
and group representations. It was also famously used by Ngô \cite{MR2653248} 
to prove the \emph{fundamental lemma}, a key result in the Langlands program.
\end{enumerate}

\cleardoublepage\phantomsection % So that the hyperref thumbnail is correct
\newgeometry{top=0.09\paperheight, bottom=0.10\paperheight, left=0.09\paperwidth, right=0.08\paperwidth}
\hypersetup{urlcolor=black}
{\raggedright \printbibliography[heading=bibintoc, title={References}]}
\restoregeometry

\end{document}